\documentclass[11pt,dvips,twoside]{article}

\usepackage{pslatex}
\usepackage{fancyhdr}
\usepackage{graphicx}
\usepackage{geometry}

\RequirePackage{amsfonts,amssymb,amsmath,amscd,amsthm}
\RequirePackage{txfonts}
\RequirePackage{graphicx}
\RequirePackage{xcolor}
\RequirePackage{geometry}
\RequirePackage{enumerate}

\def\figurename{Figure} % Replace the colon that normally appears after the Figure number by a period.
\makeatletter
\renewcommand{\fnum@figure}[1]{\figurename~\thefigure.}
\makeatother

\def\tablename{Table} % Replace the colon that normally appears after the Figure number by a period.
\makeatletter
\renewcommand{\fnum@table}[1]{\tablename~\thetable.}
\makeatother

\usepackage{amsmath}
\usepackage{amssymb}
\usepackage{amsfonts}
\usepackage{amsthm,amscd}

\newtheorem{theorem}{Theorem}[section]
\newtheorem{lemma}[theorem]{Lemma}
\newtheorem{corollary}[theorem]{Corollary}
\newtheorem{proposition}[theorem]{Proposition}
\theoremstyle{definition}
\newtheorem{definition}[theorem]{Definition}
\newtheorem{definitions}[theorem]{Definitions}
\newtheorem{example}[theorem]{Example}

\newtheorem{remark}[theorem]{Remark}
\theoremstyle{remark}

\numberwithin{equation}{section}

\begin{document}

%-----------------------------------------------------------------------------------
\title{Constructions and generalised derivations of multiplicative $n$-BiHom-Lie color algebras}
%\label{firstpage}
\author{IBRAHIMA BAKAYOKO \thanks{D\'epartement de Math\'ematiques,
Universit\'e de N'Z\'er\'ekor\'e,
BP 50 N'Z\'er\'ekor\'e, Guin\'ee. Email: ibrahimabakayoko27@gmail.com}\and ISMAIL LARAIEDH\thanks{Departement of Mathematics, Faculty of Sciences, Sfax University, BP 1171, 
3000 Sfax, Tunisia. Email:
Ismail.laraiedh@gmail.com and Departement of Mathematics, College of Sciences and Humanities - Kowaiyia, Shaqra University,
Kingdom of Saudi Arabia. E.mail:
ismail.laraiedh@su.edu.sa}}

\date{}
%--------------------------------------------------------------------
\maketitle
% ----------------------------------------------------------------

\begin{abstract}
The aim of this paper is introduce and give some constructions results and examples of $n$-BiHom-Lie color algebras. Next, we introduce the definition of BiHom-modules over $n$-BiHom-Lie color algebras and we provide some properties. Moreover we investigate generalized derivations of $n$-BiHom-Lie color algebras and their BiHom-subalgebras.

\end{abstract}
%
%
% \title{\bfseries\scshape{n-Hom-Lie color algebras}}
%
%
% \author{\bfseries\scshape Ibrahima BAKAYOKO\thanks{e-mail address: ibrahimabakayoko27@gmail.com}\\
% D\'epartement de Math\'ematiques,
% Universit\'e de N'Z\'er\'ekor\'e\\
% BP 50 N'Z\'er\'ekor\'e, Guin\'ee.
%  \\\bfseries\scshape Sergei Silvestrov \thanks{e-mail address: ssilvest@maths.lth.se}\\
% Centre for Mathematical Sciences, Lund University,\\ Box 118, SE-221 00 Lund, Sweden}
%
% \date{}
% \maketitle

% % % % % % \thispagestyle{empty} \setcounter{page}{1}
% % % % % % % ------- [First Page Running Head] - place it immediately after title! ------
% % % % % % \thispagestyle{fancy} \fancyhead{}
% % % % % % \fancyhead[L]{{\LARGE A}frican {\LARGE D}iaspora {\LARGE J}ournal of {\LARGE M}athematics\\
% % % % % % Volume X, Number X, pp. {\thepage--\pageref{lastpage-01} (2013)}} % put \label{lastpage-xx} on the last page!
% % % % % % \fancyhead[R]{ISSN 1539-854X  \\ {\tt{www.math-res-pub.org/adjm}}}
% % % % % % \fancyfoot{}
% % % % % % \renewcommand{\headrulewidth}{0pt}
%------------------------------------------------------------------------------

% \noindent\hrulefill
%
% \noindent {\bf Abstract.}
% The purpose of this paper is to generalize some results on either n-Lie algebras or n-Hom-Lie algebras  to the case of n-Hom-Lie color color algebras.
% Then we  introduce and give some constructions n-Hom-Lie color color algebras.
%
% \noindent \hrulefill

\vspace{.3in}

% \noindent {\bf AMS Subject Classification:} 17B75; 16W50.

% \vspace{.08in} \noindent \textbf{Keywords}: n-Hom-Lie color algebras, color modules, averaging, semi-morphism, morphism.
\vspace{.3in}
% \noindent {\small Revised: April 16, 2014, June 28, 2014 $\parallel$  Accepted: July 4, 2014}
\vspace{.2in}
% \newpage

% ------------ [Running Heads - for odd and even pages] - please insert it only on page 2!
% \pagestyle{fancy} \fancyhead{} \fancyhead[EC]{Ibrahima Bakayoko and Sergei Silvestrov}
% \fancyhead[EL,OR]{\thepage} \fancyhead[OC]{n-Hom-Lie color algebras} \fancyfoot{}
\renewcommand\headrulewidth{0.5pt}
%------------------------------------------------------------------------------

\section{Introduction}

% n-ary algebras in Lie and Jordan algebras, geometry, math physics, M-branes

Algebras with $n$-ary compositions play important roles in Lie and Jordan theories, geometry, analysis, and physics.  For instance, Jordan triple systems \cite{jacobson1,lister,neher} give rise to $3$-graded Lie algebras through the Kantor-Koecher-Tits construction \cite{kantor1,koecher,tits}, from which most simple Lie algebras can be obtained.  Jordan triple systems also give rise to Lie triple systems through the Meyberg construction \cite{meyberg,meyberg2}.  On the other hand, Lie triple systems give rise to $\mathbb{Z}/2\mathbb{Z}$-graded Lie algebras \cite{jacobson1,lister}, which are exactly the kind of Lie algebras associated to symmetric spaces.  In geometry and analysis, various types of Jordan triple systems are used in the classifications of different classes of symmetric spaces \cite{bertram,chu,kaup1,kaup2,loos1,loos2}.

The n-Lie algebras found their applications in many fields of Mathematics and Physics. Ternary Lie algebras appeared first in Nambu generalization of Hamiltonian mechanics \cite{Nambu:GenHD}
using ternary bracket generalization of Poisson algebras.  Nambu mechanics \cite{Nambu:GenHD} involves an $n$-ary product that satisfies the $n$-ary Nambu identity, which is an $n$-ary generalization of the Jacobi identity.  Bagger-Lambert algebras \cite{bl} are ternary Nambu algebras with some extra structures, and they appear in the study of string theory and $M$-branes.  Ternary algebras are used in \cite{gun,gh,gh2} to construct superconformal algebras.
Deformations of $n$-Lie algebras have been
studied from several aspects. See  \cite{Azc?rraga,Figueroa}
for more details. In particular, a notion of a Nijenhuis
operator on a $3$-Lie algebra was introduced in  \cite{Zhang}
in the study of the $1$-order deformations of a $3$-Lie algebra. But there are some quite strong conditions in this
definition of a Nijenhuis operator. In the case of Lie algebras, one could obtain fruitful results by considering
one-parameter infinitesimal deformations, i.e. $1$-order deformations. However, for $n$-Lie algebras, we believe that
one should consider $(n-1)$-order deformations to obtain
similar results. In \cite{Figueroa}, for $3$-Lie algebras, the author had already considered $2$-order deformations. For the case of Hom-Lie superalgebras, the authors in \cite{Liu&chen&Ma} give the notion of Hom-Nijenhuis operator.

The origin of Hom-structures can be found in the physics literature around 1900, appearing
in the study of quasideformations of Lie algebras of vector fields. n-ary Hom-type generalization of n-ary algebras
were introduced in \cite{H L S}. Derivations and generalized derivations of many varieties of
algebras and Hom-algebras were investigated in \cite{A M A, K2, L L, K N, Z K, Z C M, Z Z}. Then, hom-Lie algebras were generalized to hom-Lie superalgebras by Ammar
and Makhlouf \cite{F Ab, F Ab N}. Hom-algebra structures of a given type include their classical counterparts and open more possibilities for deformations, Hom-algebra extensions of cohomological structures and representations, formal deformations of Hom-associative and Hom-Lie algebras,
Hom-Lie admissible Hom-coalgebras, Hom-coalgebras, Hom-Hopf algebras \cite{AmmarEjbehiMakhlouf:homdeformation,BenMakh:Hombiliform,LarssonSilvJA2005:QuasiHomLieCentExt2cocyid,MakhSil:HomHopf,MakhSilv:HomAlgHomCoalg,MakhSilv:HomDeform,Sheng:homrep,Yau:HomHom,Yau:HomEnv}.

A generalization of this approach led the authors of \cite{GRAZIANI} to introduce BiHom-algebras, which are algebras where the identities defining the structure are twisted by two homomorphisms $\alpha$ and $\beta$.
This class of algebras can be viewed as an extension of the class of Hom-algebras since, when the two linear maps of a BiHom-algebra are the same, it reduces to a Hom-algebra. These
algebraic structures include BiHom-associative algebras, BiHom-Lie algebras and BiHom-bialgebras. More applications of BiHom-algebras, BiHom-Lie superalgebras and BiHom-Novikov algebras can be
found in \cite{S S, Guo1, chtioui1,chtioui2, ismail, Li C, Liu2}. BiHom-type generalizations of n-ary Nambu-Lie algebras, called n-ary BiHom-Nambu-
Lie algebras, were introduced by Kitouni, Makhlouf, and Silvestrov in [12]. Each n-ary
BiHom-Nambu-Lie algebra has (n - 1)-linear twisting maps, which appear in a twisted
generalization of the n-ary Nambu identity called the n-ary BiHom-Nambu identity.

In 1979, Scheunert investigated the Lie color algebras from a purely mathematical
point of view and obtained generalizations of the PBW and Ado theorems( \cite{sch1}). The
cohomology theory for Lie color algebras was introduced in \cite{sch2}. Moreover, various properties were studied in color setting, can be
found in \cite{elka, elkadri, FA, Bak}.

The aim of this paper is to extend the work done By I. Bakayoko and S Silvestrov in \cite{Bakayoko} in BiHom-algebras setting. The paper is organized as follows. In Section 2, summarizes basic notions and notations related to $n$-BiHom-Lie color algebras which will be used in
next sections. In Section 3, we provide some useful methods for construction of n-BiHom-Lie
color algebras. In Section 4, we consider BiHom-modules over $n$-BiHom-Lie color algebras and give some results. In Section 5, we introduce and study proper-
ties of derivations, $(\alpha^{k},\beta^{r})$-derivations and generalized derivations of $n$-BiHom-Lie color algebras
and their Bihom-subalgebras.

Throughout this paper, all graded linear spaces are assumed to be over a field $\mathbb{K}$ of
characteristic different from $2$.

\section{Preliminaries}
\label{BakayokoSilvestrov:sec:preliminaries}
In the following we summarize definitions and proprieties of $n$-BiHom-Lie color algebra.\\

Let $\Gamma$ be an abelian group. A vector space $V$ is said to be $\Gamma$-graded, if there is a family
$(V_{\gamma})_{\gamma\in \Gamma}$ of vector subspace of $V$ such that $$V=\bigoplus_{\gamma\in \Gamma}V_{\gamma}.$$
An element $x \in V$ is said to be homogeneous of  degree $\gamma \in \Gamma$ if $x \in V_{\gamma}, \gamma\in \Gamma$, and in this case, $\gamma$ is called the degree of $x$. As usual, we denote by $\overline{x}$ the degree of an element $x \in V$. Thus each homogeneous
element $x \in V$ determines a unique group of element  $\overline{x} \in \Gamma$ by $x \in V_{\overline{x}}$.

Let $V=\bigoplus_{\gamma\in \Gamma}V_{\gamma}$ and $V^{'}=\bigoplus_{\gamma\in \Gamma}V^{'}_{\gamma}$ be two $\Gamma$-graded vector spaces. A linear mapping $f: V \longrightarrow V^{'}$  is said to be homogeneous of  degree $\upsilon \in \Gamma$ if
  $f(V_{\gamma})\subseteq V^{'}_{\gamma+\upsilon}, ~~ \forall \gamma \in \Gamma.$
If in addition $f$ is  homogeneous of degree zero, i.e. $f(V_{\gamma})\subseteq V^{'}_{\gamma}$ holds for any $\gamma \in \Gamma$, then $f$ is said to be even.

An algebra $A$ is said to be $\Gamma$-graded if its underlying vector space is $\Gamma$-graded, i.e. $A=\bigoplus_{\gamma\in \Gamma}A_{\gamma}$, and if, furthermore $A_{\gamma}A_{\gamma'}\subseteq A_{\gamma+\gamma'}$, for all $\gamma, \gamma'\in \Gamma$. It is easy to see
that if $A$ has a unit element $e$, it follows that $e \in A_{0}$. A subalgebra of $A$ is said to be graded if it is graded as a subspace of $A$.\\
\begin{definition} Let $\mathbb{K}$ be a field and  $\Gamma$ be an abelian group. A map $\varepsilon:\Gamma\times\Gamma\rightarrow \mathbb{K^{\ast}}$ is called a skewsymmetric \textit{bicharacter} on ${\Gamma}$ if the following identities hold, for all $a,b,c$ in $\Gamma$
\begin{enumerate}
\item~~$\varepsilon(a,b)~\varepsilon(b,a)=1,$
\item~~$\varepsilon(a,b+c)=\varepsilon(a,b)~\varepsilon(a,c),$
\item~~$\varepsilon(a+b,c)=\varepsilon(a,c)~\varepsilon(b,c).$
\end{enumerate}
\end{definition}
 The definition above implies, in particular, the following relations
$$\varepsilon(a,0)=\varepsilon(0,a)=1,\ \varepsilon(a,a)=\pm1, \  \textrm{for\ all}\  a \in \Gamma.$$
If $x$ and $x'$ are two homogeneous elements of degree $\gamma$ and $\gamma'$ respectively and $\varepsilon$ is a skewsymmetric bicharacter, then we shorten the notation by writing $\varepsilon(x,x')$ instead of $\varepsilon(\gamma,\gamma')$.

\begin{example} Some standard examples of skew-symmetric bicharacters are:
 \begin{enumerate}
\item [1)] $\Gamma=\mathbb{Z}_2,\quad \varepsilon(i, j)=(-1)^{ij}$,
\item [2)] $\Gamma=\mathbb{Z}_2^n=\{(\alpha_1, \dots, \alpha_n)|~ \alpha_i\in\mathbb{Z}_2 \}, ~~
\varepsilon((\alpha_1, \dots, \alpha_n), (\beta_1, \dots, \beta_n)):= (-1)^{\alpha_1\beta_1+\dots+\alpha_n\beta_n}.$
\item [3)] $\Gamma=\mathbb{Z}_2\times\mathbb{Z}_2,\quad \varepsilon((i_1, i_2), (j_1, j_2))=(-1)^{i_1j_2-i_2j_1}$,
\item [4)] $\Gamma=\mathbb{Z}\times\mathbb{Z} ,\quad \varepsilon((i_1, i_2), (j_1, j_2))=(-1)^{(i_1+i_2)(j_1+j_2)}$,
\item [5)] $\Gamma=\{-1, +1\} , \quad\varepsilon(i, j)=(-1)^{(i-1)(j-1)/{4}}$.
% \item [vi)]
\end{enumerate}
\end{example}
\begin{definition}
An $ n $-Lie color algebra is a triple $(L,[\cdot,\dots,\cdot],\varepsilon)$ consisting of a $\Gamma$-graded vector space $L$, an even $ n $-linear map
 $[\cdot \dots, \cdot]: L\times\dots \times L\rightarrow L$, a bicharacter $\varepsilon : \Gamma\times \Gamma\rightarrow \mathbb{K}^*$ such that:
	\begin{enumerate}
\item
$[x_{1},\ldots,x_{k},x_{k+1},\ldots,x_{n-1},x_{n}]
=-\varepsilon(x_{k},x_{k+1})[x_{1},\ldots,x_{k+1},x_{k},\ldots,x_{n-1},x_{n}]$,
$ \quad \quad  k=1,2, \dots, n-1.$ \nonumber

\item
\begin{equation}
\begin{array}{l}
 [x_1, \dots, x_{n-1}, [y_1, \dots, y_n]]
= \displaystyle \sum_{k=1}^{n}\varepsilon(X, Y_k) [y_1, \dots, y_{k-1}, [x_1, \dots, x_{n-1}, y_k], y_{k+1}, \dots y_{n}],\\
\end{array} \label{BakayokoSilvestrov:F}
\end{equation}
for all $ x_{1},\ldots,x_{n-1},y_{1},\ldots,y_{n} \in L,~~X=\sum_{k=1}^{n-1}x_k,~  Y_k=\sum_{l=1}^ky_{l-1}$ and $y_0=e$.
	\end{enumerate}
\end{definition}
%\begin{definition}
% An $n$-Hom-Lie color algebra \index{Algebra!n-Hom-Lie color} is a graded linear space $L=\oplus L_a, a\in \Gamma$ with an $n$-linear map
% $[\cdot \dots, \cdot]: L\times\dots \times L\rightarrow L$, a bicharacter $\varepsilon : \Gamma\times \Gamma\rightarrow \mathbb{K}^*$ and an even linear map
%$\alpha : L\rightarrow L$ such that
%\begin{eqnarray}
%[x_1, \dots, x_k, x_{k+1}, \dots, x_n]&=&-\varepsilon(x_k, x_{k+1})[x_1, \dots, x_{k+1}, x_k, \dots, x_n], \\
%&& \quad \quad \quad \quad \quad \quad  k=1,2, \dots n-1. \nonumber
%\end{eqnarray}
%\begin{equation} %
%\begin{array}{l}
%[\alpha(x_1), \dots, \alpha(x_{n-1}), [y_1, y_2, \dots, y_{n}]] = \\
%=\displaystyle \sum_{k=1}^n\varepsilon(X, Y_k)[\alpha(y_1), \dots, \alpha(y_{k-1}), [x_1, \dots, x_{n-1}, y_k], \alpha(y_{k+1}),\dots, \alpha(y_{n})]
%\end{array}
%\end{equation}
%where $x_k, y_l\in\mathcal{H}(L)$, $X=\sum_{k=1}^{n-1}x_i$,  $Y_k=\sum_{l=1}^ky_{l-1}$ and $y_0=e$.
%\end{definition}

\begin{example}
 Let $\Gamma=\mathbb{Z}_2,\quad \varepsilon(i, j)=(-1)^{ij}$,
$L=L_0\oplus L_1=<e_2, e_4>\oplus <e_1, e_3>$,
$$[e_1, e_2, e_3]=e_2,\quad [e_1, e_2, e_4]=e_1,\quad [e_1, e_3, e_4]=[e_2, e_3, e_4]=0.$$
Then $(L, [\cdot, \cdot, \cdot], \varepsilon)$ is a $3$-Lie color algebra.
\end{example}

\begin{definition} A $3$-BiHom-Lie color algebra is a $5$-tuple $(L,[\cdot,\cdot,\cdot],\varepsilon,\alpha,\beta)$ consisting of a $\Gamma$-graded vector space $L$, an even trilinear mapping
$[\cdot,\cdot,\cdot]:L\times L\times L\rightarrow L$, a bicharacter $\varepsilon:\Gamma\times\Gamma\rightarrow \mathbb{K^{\ast}}$ and two even linear maps $\alpha,\beta:L\rightarrow L$ such that for homogeneous elements $x,y,z,u,v$ we have
\begin{enumerate}
  \item $\alpha \circ \beta = \beta \circ \alpha,$
  \item $[\beta(x),\beta(y),\alpha(z)]=-\varepsilon(x,y)[\beta(y),\beta(x),\alpha(z)]=-\varepsilon(y,z)[\beta(x),\beta(z),\alpha(y)],\\$($\varepsilon-3-$BiHom-skewsymmetry).
  \item
  $ $\\
  $\begin{array}{lllllll}[\beta^{2}(x),\beta^{2}(y),[\beta(z),\beta(u),\alpha(v)]]&=&\varepsilon(u+v,x+y+z)[\beta^{2}(u),\beta^{2}(v),[\beta(x),\beta(y),\alpha(z)]]\\
&=&-\varepsilon(z+v,x+y)\varepsilon(u,v)[\beta^{2}(z),\beta^{2}(v),[\beta(x),\beta(y),\alpha(u)]]\\
&=&+\varepsilon(z+u,x+y)[\beta^{2}(z),\beta^{2}(u),[\beta(x),\beta(y),\alpha(v)]],\end{array}$

($\varepsilon-3-$BiHom-Jacobi identity).
\end{enumerate}
\end{definition}

\begin{definition}
	An $ n $-BiHom-Lie color algebra  A is a $5$-tuple $(L,[.,\dots,.],\varepsilon,\alpha,\beta)$ consisting of a $\Gamma$-graded vector space $L$, an even $ n $-linear map
 $[\cdot \dots, \cdot]: L\times\dots \times L\rightarrow L$, a bicharacter $\varepsilon : \Gamma\times \Gamma\rightarrow \mathbb{K}^*$ and two even linear maps $ \alpha $ and $ \beta $ such that :
	\begin{enumerate}
		\item $ \alpha \circ \beta=\beta \circ \alpha $.
		%		\item $ \alpha $ and $ \beta $ are morphisms of $ n $-algebras.
\item
$  $

$
\begin{array}{lllll}
&&[\beta(x_{1}),\ldots,\beta(x_{k}),\beta(x_{k+1}),\ldots,\beta(x_{n-1}),\alpha(x_{n})]\\
&=&-\varepsilon(x_{k},x_{k+1})[\beta(x_{1}),\ldots,\beta(x_{k+1}),\beta(x_{k}),\ldots,\beta(x_{n-1}),\alpha(x_{n})]\\
&=&-\varepsilon(x_{n-1},x_{n})[\beta(x_{1}),\ldots,\beta(x_{n-2}),\beta(x_{n}),\alpha(x_{n-1})],\end{array}$

where $k=1,2, \dots , n-2.$ \nonumber

		\item
		\begin{align*}
			& [\beta^{2}(x_{1}),\ldots,\beta^{2}(x_{n-1}),[\beta(y_{1}),\ldots,\beta(y_{n-1}),\alpha(y_{n})]] \\ =&  \sum_{k=1}^{n}  \varepsilon(X, Y_k) [\beta^{2}(y_{1}),\ldots,\beta^{2}(y_{k-1}),[\beta(x_{1}),\ldots, \beta(x_{n-1}),\alpha(y_{k})],\beta^{2}(y_{k+1}),\ldots,\beta^{2}(y_{n})] ,
		\end{align*}  for all $ x_{1},\ldots,x_{n-1},y_{1},\ldots,y_{n} \in L,~~X=\sum_{k=1}^{n-1}x_k,~  Y_k=\sum_{l=1}^ky_{l-1}$ and $y_0=e$.
	\end{enumerate}
\end{definition}

\begin{remark}
\begin{enumerate}
\item [1)] The $n$-BiHom $\varepsilon$-Jacobi identity is equivalent to
\begin{align*}
			& [\beta^{2}(x_{1}),\ldots,\beta^{2}(x_{n-1}),[\beta(y_{1}),\ldots,\beta(y_{n-1}),\alpha(y_{n})]] \\ =&  \sum_{k=1}^{n}  \varepsilon(X, \widetilde{Y}_n)\varepsilon(y_{k},\overline{Y}_{k+1}) [\beta^{2}(y_{1}),\ldots,\beta^{2}(y_{k-1}),\beta^{2}(y_{k+1}),\ldots,\beta^{2}(y_{n}),[\beta(x_{1}),\ldots, \beta(x_{n-1}),\alpha(y_{k})]] ,
		\end{align*}  for all $ x_{1},\ldots,x_{n-1},y_{1},\ldots,y_{n} \in L,~~X=\sum_{k=1}^{n-1}x_k,~  \widetilde{Y}_n=\displaystyle{\sum_{l=1~l\neq k}^n} y_{l}$ and $\overline{Y}_{k+1}=\sum_{i=k+1}^{n}y_{i}$.
\item [2)] When $\alpha=\beta=id$, we get $n$-Lie color algebra.
\item [3)] When $\Gamma=\{e\}$ and $\alpha=\beta=id$, we get $n$-Lie algebra.
\item [4)] When $\Gamma=\{e\}$ and $\alpha,\beta\neq id$, we get $n$-BiHom-Lie algebra.
\end{enumerate}
\end{remark}
\begin{definitions}
\begin{enumerate}
\item An $n$-BiHom-Lie color algebra $(L,[\cdot, \dots, \cdot],\varepsilon,\alpha,\beta)$ is multiplicative if for any homogenous elements $x_{1},\cdots,x_{n}\in L,$ we have
$$
\alpha([x_{1},\cdots,x_{n}])=[\alpha(x_{1}),\cdots,\alpha(x_{n})]~~~~~~\text{and}~~~~\beta([x_{1},\cdots,x_{n}])=[\beta(x_{1}),\cdots,\beta(x_{n})].
$$
\item An $n$-BiHom-Lie color algebra $(L,[\cdot, \dots, \cdot],\varepsilon,\alpha,\beta)$ is regular if $\alpha$ and $\beta$ are  even algebra automorphisms.
\item An $n$-BiHom-Lie color algebra $(L,[\cdot, \dots, \cdot],\varepsilon,\alpha,\beta)$ is involutive if $\alpha$ and $\beta$ are involutions, i.e., $\alpha^{2}=\beta^{2}=id$.
\end{enumerate}
\end{definitions}

\begin{definition}
 A morphism   $f : (L, [\cdot, \dots, \cdot], \varepsilon, \alpha,\beta)\rightarrow (L', [\cdot, \dots, \cdot]', \varepsilon, \alpha',\beta')$ of an $n$-BiHom-Lie color algebras
 is an even linear map $f : L\rightarrow L'$ such that
$f\circ\alpha=\alpha'\circ f,~~f\circ\beta=\beta'\circ f$  and for any $x_k\in\mathcal{H}(L)$,
$$f([x_1, \dots, x_n])=[f(x_1), \dots, f(x_n)]'$$
\end{definition}
Denote by $\phi_f = \{x+f(x); ~~x\in L\}\subset L\oplus L'$ which is the graph of a linear map $f : L\rightarrow L'$.
\begin{definitions}
\begin{enumerate}
\item
A subspace $H$ of $L$ is an $n$-BiHom-subalgebra of $(L,[.,\cdots,.],\varepsilon,\alpha,\beta)$ if $\alpha(H)\subset H$,
$\beta(H)\subset H$ and $[H,\cdots,H]\subseteq H$, (i.e.,
$[x_{1},\cdots,x_{n}]\in H,~~\forall x_{k}\in H$).
\item
A subspace $I\subset L$ is an $n$-BiHom ideal of $(L,[.,\cdots,.],\varepsilon,\alpha,\beta)$ if $\alpha(I)\subset I$, $\beta(I)\subset I$\\
and $[I,L,\cdots,L]\subseteq I$, (i.e.
$[x,y_{1},\cdots,y_{n-1}]\in I,~~\forall x\in I;~y_{k}\in L$.
\end{enumerate}
\end{definitions}
\begin{definitions}
\begin{enumerate}
\item
The center of $(L,[.,\cdots,.],\varepsilon,\alpha,\beta)$ is the set of $x\in L$ such that\\ $[x,y_{1},\cdots,y_{n-1}]=0,$ for any $y_{i}\in L$. The center is an ideal of $L$ which we will denote by $Z(L)$.
\item
The $(\alpha,\beta)$-center of $(L,[.,\cdots,.],\varepsilon,\alpha,\beta)$ is the set
$$Z_({\alpha,\beta})=\{x\in L,~[x,\alpha\beta(y_{1}),\cdots,\alpha\beta(y_{n-1})]=0,~~for~any~y_{k}\in L\}$$
\end{enumerate}

\end{definitions}

\begin{definition}
Let $H_1, H_2, \dots, H_n$ be BiHom-subalgebras of an $n$-BiHom-Lie color algebra $L$. Denote by $[H_1, H_2, \dots, H_n]$ the BiHom-subalgebra of $L$
generated by all elements $[x_1, x_2, \dots, x_n]$, where $x_k\in H_k, k=1, 2, \dots, n$.
\begin{enumerate}
 \item [i)] The sequence $L_1, L_2, \dots, L_n, \dots$ defined by
\begin{gather*}
L_0=L,\quad L_1=[L_0, L_0, \dots, L_0],\quad L_2=[L_1, L_1, \dots, L_1], \dots,\\
L_n=[L_{n-1}, L_{n-1}, \dots, L_{n-1}], \dots
\end{gather*}
is called the derived sequence. \index{Derived sequence}
\item [ii)] The sequence $L^1, L^2, \dots, L^n, \dots$ defined by
\begin{gather*} L^0=L,\quad L^1=[L^0, L, \dots, L],\quad L^2=[L^1, L, \dots, L], \dots, \\
L^n=[L^{n-1}, L, \dots, L], \dots \end{gather*}
is called the descending central sequence. \index{Descending central sequence}
\end{enumerate}
\end{definition}
\begin{theorem}
 Let $(L, [\cdot, \dots, \cdot], \varepsilon, \alpha,\beta)$ be an involutive $n$-BiHom-Lie color algebra. Then,
$I_n, I^n$ and $Z(L)$ are BiHom-ideals of $L$.
\end{theorem}
\begin{proof}
We only prove, by induction, that $I_n$ is a BiHom-ideal. For this, suppose, first, that $I_{n-1}$ is a BiHom-subalgebra of $L$ and show that $I_n$
is a BiHom-subalgebra of $L$.
For any $y\in \mathcal{H}(I_n)$, there exist $y_1, y_2, \dots, y_n\in \mathcal{H}(I_{n-1})$, such that
$$y=[y_1, y_2, \dots, y_n].$$ So, $\alpha(y)=\alpha([y_1, y_2, \dots, y_n])=\alpha(y_1), \alpha(y_2), \dots, \alpha(y_n)]\in I_{n}$ and $\beta(y)=\beta([y_1, y_2, \dots, y_n])=$\\$[[\beta(y_1), \beta(y_2), \dots, \beta(y_n)]\in I_{n}$, because $I_{n-1}$ is a BiHom-subalgebra. That is $\alpha(I_n)\subseteq I_n$ and $\beta(I_n)\subseteq I_n$.

For any $y_k\in \mathcal{H}(I_{n})$, there exist $y_k^1, y_k^2, \dots, y_k^n\in I_{n-1}, k=1, 2, \dots, n$ such that
$$[y_1, y_2, \dots, y_n]=[[y_1^1, y_1^2, \dots, y_1^n], [y_2^1, \dots, y_2^n], \dots, [y_n^1, \dots, y_n^n]].$$
$I_{n-1}$ being a BiHom-subalgebra, by hypotheses, $[y_k^1, y_i^2, \dots, y_k^n]\in I_{n-1}$ for $1\leq i\leq n$, and so
$[y_1, y_2, \dots, y_n]\in I_n$. Thus $I_n$ is a BiHom-subalgebra.

Now, suppose that $I_{n-1}$ is a BiHom-ideal. Let $x_1, \dots, x_{n-1}\in L, y\in I_n$, then there exist $y_1, \dots, y_{n}\in I_{n-1}$ (i.e $\alpha$ and $\beta$ are involution) such that
$$\begin{array}{lllllll}
&& [x_1, \dots, x_{n-1}, y]
=[\beta^{2}(x_1), \dots, \beta^{2}(x_{n-1}), [\beta(y_1), \dots, \beta(y_{n-1}),\alpha(y_{n})]]\nonumber\\
&&=\sum_{k=1}^{n}  \varepsilon(X, Y_k) [\beta^{2}(y_{1}),\ldots,\beta^{2}(y_{k-1}),[\beta(x_{1}),\ldots, \beta(x_{n-1}),\alpha(y_{k})],\beta^{2}(y_{k+1}),\ldots,\beta^{2}(y_{n})]
\end{array}$$
As $[\beta(x_{1}),\ldots, \beta(x_{n-1}),\alpha(y_{k})]\in I_{n-1}$, then $[x_1, \dots, x_{n-1}, y]\in I_n$. So, $I_n$ is a BiHom-ideal of $L$.
\end{proof}

\section{Constructions of $n$-BiHom-Lie color algebras}
\label{BakayokoSilvestrov:sec:constrnhomcoloralg}
In this section we present some useful methods for construction of $n$-BiHom-Lie color algebras.
\begin{proposition}
Let $(L, [\cdot, \dots, \cdot], \varepsilon, \alpha,\beta)$ be an $n$-BiHom-Lie color algebra and $I$ a BiHom-ideal of $(L, [\cdot, \dots, \cdot], \varepsilon, \alpha,\beta)$.
Then $(L/I,\overline{[\cdot, \dots, \cdot]}, \varepsilon,\overline{\alpha},\overline{\beta})$ is an $n$-BiHom-Lie color algebra where $[\overline{x_{1}}, \dots, \overline{x_{n}}]=\overline{[x_{1}, \dots, x_{n}]}$,
$\overline{\alpha}(\overline{x})=\overline{\alpha(x)}$  and $\overline{\beta}(\overline{x})=\overline{\beta(x)}$, for all $\overline{x},~\overline{x_{1}},\ldots,\overline{x_{n}}\in \mathcal{H}(A/I)$
\end{proposition}
\begin{proof}
It follows from a straightforward computation.
\end{proof}
\begin{proposition}
Let $(L, [\cdot, \dots, \cdot], \varepsilon, \alpha,\beta)$ be an $n$-BiHom-Lie color algebra and $u\in L_e$ such that $\beta(u)=u$.
Then $(L, \{\cdot, \dots, \cdot\}, \varepsilon, \alpha,\beta)$ is an $(n-1)$-BiHom-Lie color algebra with
$$\{x_1, \dots, x_{n-1}\}=[u, x_1, \dots, x_{n-1}].$$
\end{proposition}
\begin{proof} With conditions in the statement, for all $x_{i},y_{i}\in L$, we have
$$
\begin{array}{lllll}
&&\{\beta(x_{1}),\ldots,\beta(x_{k}),\beta(x_{k+1}),\ldots,\beta(x_{n-2}),\alpha(x_{n-1})\}\\[0.2cm]
&=&[u,\beta(x_{1}),\ldots,\beta(x_{k}),\beta(x_{k+1}),\ldots,\beta(x_{n-2}),\alpha(x_{n-1})]\\[0.2cm]
&=&[\beta(u),\beta(x_{1}),\ldots,\beta(x_{k}),\beta(x_{k+1}),\ldots,\beta(x_{n-2}),\alpha(x_{n-1})]\\[0.2cm]
&=&-\varepsilon(x_{k},x_{k+1})[\beta(u),\beta(x_{1}),\ldots,\beta(x_{k+1}),\beta(x_{k}),\ldots,\beta(x_{n-2}),\alpha(x_{n-1})]\\[0.2cm]
&=&-\varepsilon(x_{n-1},x_{n-2})[\beta(u),\beta(x_{1}),\ldots,\beta(x_{n-3}),\beta(x_{n-1}),\alpha(x_{n-2})]\\[0.2cm]
&=&-\varepsilon(x_{k},x_{k+1})[u,\beta(x_{1}),\ldots,\beta(x_{k+1}),\beta(x_{k}),\ldots,\beta(x_{n-2}),\alpha(x_{n-1})]\\[0.2cm]
&=&-\varepsilon(x_{n-1},x_{n-2})[u,\beta(x_{1}),\ldots,\beta(x_{n-3}),\beta(x_{n-1}),\alpha(x_{n-2})]\\[0.2cm]
&=&-\varepsilon(x_{k},x_{k+1})\{\beta(x_{1}),\ldots,\beta(x_{k+1}),\beta(x_{k}),\ldots,\beta(x_{n-2}),\alpha(x_{n-1})\}\\[0.2cm]
&=&-\varepsilon(x_{n-1},x_{n-2})\{\beta(x_{1}),\ldots,\beta(x_{n-3}),\beta(x_{n-1}),\alpha(x_{n-2})\},\end{array}$$ where  k=1,2, \dots n-3.

Next, we have
$$\begin{array}{llll}
&&\{\beta^{2}(x_1), \dots, \beta^{2}(x_{n-2}), \{\beta(y_1), \dots, \beta(y_{n-2}), \alpha(y_{n-1})\}\} \nonumber\\[0.2cm]
&=&[u,\beta^{2}(x_1), \dots, \beta^{2}(x_{n-2}), [u,\beta(y_1), \dots, \beta(y_{n-2}), \alpha(y_{n-1})]]\nonumber\\[0.2cm]
&=&[\beta^{2}(u),\beta^{2}(x_1), \dots, \beta^{2}(x_{n-2}), [\beta(u),\beta(y_1), \dots, \beta(y_{n-2}), \alpha(y_{n-1})]]\nonumber\\[0.2cm]

&=&[[\beta(u), \beta(x_1), \dots, \beta(x_{n-2}), \alpha(u)], \beta^{2}(y_1), \dots, \beta^{2}(y_{n-1})]\nonumber\\[0.2cm]
&&+\displaystyle{\sum_{k=1}^{n-1}}\varepsilon(X, Y_k)[\beta^{2}(u), \beta^{2}(y_1), \dots, \beta^{2}(y_{k-1}), [\beta(u), \beta(x_1), \dots, \beta(x_{n-2}), \alpha(y_k)], \beta^{2}(y_{k+1}), \dots, \beta^{2}(y_{n-1})]\nonumber\\[0.2cm]

&=&[[u, \beta(x_1), \dots, \beta(x_{n-2}), u], \beta^{2}(y_1), \dots, \beta^{2}(y_{n-1})]\nonumber\\[0.2cm]
&&+\displaystyle{\sum_{k=1}^{n-1}}\varepsilon(X, Y_k)[u, \beta^{2}(y_1), \dots, \beta^{2}(y_{k-1}), [u, \beta(x_1), \dots, \beta(x_{n-2}), \alpha(y_k)], \beta^{2}(y_{k+1}), \dots, \beta^{2}(y_{n-1})]\nonumber\\[0.2cm]
&=&\displaystyle{\sum_{k=1}^{n-1}}\varepsilon(X, Y_k)\{\beta^{2}(y_1), \dots, \beta^{2}(y_{k-1}), \{\beta(x_1), \dots, \beta(x_{n-2}), \alpha(y_k)\}, \beta^{2}(y_{k+1}), \dots, \beta^{2}(y_{n-1})\}\nonumber,
  \end{array}$$
where $X=\sum_{k=1}^{n-2}x_k,~  Y_k=\sum_{l=1}^ky_{l-1}$.

which completes the proof.
\end{proof}

\begin{corollary}
Let $(L, [\cdot, \dots, \cdot], \varepsilon, \alpha,\beta)$ be an $n$-BiHom-Lie color algebra and $u_i\in L_e$ such that $\beta(u_i)=u_i,~ i=1,2, \dots, k$.
Then $L_k=(L, \{\cdot, \dots, \cdot\}_k, \varepsilon, \alpha,\beta)$ is an $(n-k)$-BiHom-Lie color algebra with
$\{x_1, \dots, x_{n-k}\}_k=[u_1, \dots, u_k, x_1\dots, \dots, x_{n-k}].$
\end{corollary}

\begin{corollary}
Let $(L, [\cdot, \dots, \cdot], \varepsilon)$ be an $n$-Lie color algebra and $u\in L_e$.

Then $(L, \{\cdot, \dots, \cdot\}, \varepsilon)$ is an $(n-1)$-Lie color algebra with
$$\{x_1, \dots, x_{n-1}\}=[u, x_1, \dots, \dots, x_{n-1}].$$
\end{corollary}

\begin{theorem}\label{yyt}
Let $(L, [\cdot, \dots, \cdot], \varepsilon, \alpha,\beta)$ be an $n$-BiHom-Lie color algebra and $\alpha',\beta'$ be two even endomorphisms of $L$ and any two of the maps $\alpha,\beta,\alpha',\beta'$ commute. Then
$L_{\alpha',\beta'}=(L, \{\cdot, \dots, \cdot\}=[\cdot, \dots, \cdot]\circ(\alpha'\otimes\dots\otimes\alpha'\otimes \beta'), \varepsilon, \alpha\circ\alpha',\beta\circ\beta')$
is an $n$-BiHom-Lie color algebra.
Moreover suppose that $(A, [\cdot, \dots, \cdot]', \varepsilon, \gamma,\delta)$ is another $n$-BiHom-Lie color algebra and $\gamma',\delta'$ be an even endomorphism of $A$ and any two of the maps $\gamma,\delta,\gamma',\delta'$ commute.
If $f : (L, [\cdot, \dots, \cdot], \varepsilon, \alpha,\beta)\rightarrow (A, [\cdot, \dots, \cdot]', \varepsilon, \gamma,\delta)$
is a morphism such that $f\alpha'=\gamma'f$ and $f\beta'=\delta'f$, then $f : L_{\alpha',\beta'}\rightarrow A_{\gamma',\delta'}$ is also a morphism.
\end{theorem}
\begin{proof} First part is proved as follows:

We prove the $n$-BiHom-$\varepsilon$-skewsymmetry. For all $x_{i},~y_{i}\in L$ we have
$$\begin{array}{llllllll}
 &&\{\beta\beta'( x_{1}),\ldots,\beta\beta'(x_{k}),\beta\beta'(x_{k+1}),\ldots,\beta\beta'(x_{n-1}),\alpha\alpha'(x_{n})\}\\[0.2cm]
 &=&[\alpha'\beta\beta'( x_{1}),\ldots,\alpha'\beta\beta'(x_{k}),\alpha'\beta\beta'(x_{k+1}),\ldots,\alpha'\beta\beta'(x_{n-1}),\beta'\alpha\alpha'(x_{n})]\\[0.2cm]
 &=&\alpha'\beta'\Big([\beta(x_{1}),\ldots,\beta(x_{k}),\beta(x_{k+1}),\ldots,\beta(x_{n-1}),\alpha(x_{n})]\Big)\\[0.2cm]
 &=&-\alpha'\beta'\varepsilon(x_{k},x_{k+1})\Big([\beta(x_{1}),\ldots,\beta(x_{k+1}),\beta(x_{k}),\ldots,\beta(x_{n-1}),\alpha(x_{n})]\Big)\\[0.2cm]
  &=&-\varepsilon(x_{k},x_{k+1})[\alpha'\beta\beta'( x_{1}),\ldots,\alpha'\beta\beta'(x_{k+1}),\alpha'\beta\beta'(x_{k}),\ldots,\alpha'\beta\beta'(x_{n-1}),\beta'\alpha\alpha'(x_{n})]\\[0.2cm]
   &=&-\varepsilon(x_{k},x_{k+1})\{\beta\beta'( x_{1}),\ldots,\beta\beta'(x_{k+1}),\beta\beta'(x_{k}),\ldots,\beta\beta'(x_{n-1}),\alpha\alpha'(x_{n})\}.\\[0.2cm]
\end{array}$$

Similarly,

$
\{\beta\beta'( x_{1}),\ldots,\beta\beta'(x_{n-1}),\alpha\alpha'(x_{n})\}
  =-\varepsilon(x_{n-1},x_{n})\{\beta\beta'( x_{1}),\ldots,\beta\beta'(x_{n}),\alpha\alpha'(x_{n-1})\}.
$

Now, we prove the $n$-BiHom $\varepsilon-$Jacobi identity. For all $x_{i},~y_{i}\in L$
$$\begin{array}{llllllll}
&&\{(\beta\beta')^{2}(x_1), \dots, (\beta\beta')^{2} (x_{n-1}), \{\beta\beta'(y_1), \dots, \alpha\alpha'(y_n)\}\}\\[0.2cm]
 &=&\{\beta^{2}\beta'^{2}(x_1), \dots, \beta^{2}\beta'^{2} (x_{n-1}), \{\beta\beta'(y_1), \dots, \alpha\alpha'(y_n)\}\}\\[0.2cm]
&=& [\alpha'\beta^{2}\beta'^{2}(x_1), \dots, \alpha'\beta^{2}\beta'^{2} (x_{n-1}), [\alpha'\beta\beta'^{2}(y_1), \dots, \alpha'\beta\beta'^{2}(y_{n-1}),\alpha\alpha'\beta'^{2}(y_n)]]\\[0.2cm]
&=&\alpha'\beta'^{2} \Big([\beta^{2}(x_{1}),\ldots,\beta^{2}(x_{n-1}),[\beta(y_{1}),\ldots,\beta(y_{n-1}),\alpha(y_{n})]]\Big)\\[0.2cm]
&=&\alpha'\beta'^{2}\Big(\displaystyle{\sum_{k=1}^{n}}  \varepsilon(X, Y_k) [\beta^{2}(y_{1}),\ldots,\beta^{2}(y_{k-1}),[\beta(x_{1}),\ldots, \beta(x_{n-1}),\alpha(y_{k})],\beta^{2}(y_{k+1}),\ldots,\beta^{2}(y_{n})]\Big)\\[0.2cm]
&=&\displaystyle{\sum_{k=1}^{n}}  \varepsilon(X, Y_k) \alpha'\beta'^{2}\Big([\beta^{2}(y_{1}),\ldots,\beta^{2}(y_{k-1}),[\beta(x_{1}),\ldots, \beta(x_{n-1}),\alpha(y_{k})],\beta^{2}(y_{k+1}),\ldots,\beta^{2}(y_{n})]\Big)\\[0.2cm]
&=&\displaystyle{\sum_{k=1}^{n}}  \varepsilon(X, Y_k) \Big([\alpha'\beta'^{2}\beta^{2}(y_{1}),\ldots,\alpha'\beta'^{2}\beta^{2}(y_{k-1}),\alpha'\beta'^{2}[\beta(x_{1}),\ldots, \beta(x_{n-1}),\alpha(y_{k})],\\&&\alpha'\beta'^{2}\beta^{2}(y_{k+1}),\ldots,\alpha'\beta'^{2}\beta^{2}(y_{n})]\Big)\\[0.2cm]
&=&\displaystyle{\sum_{k=1}^{n}}  \varepsilon(X, Y_k) \Big([\alpha'(\beta'^{2}\beta^{2}(y_{1})),\ldots,\alpha'(\beta'^{2}\beta^{2}(y_{k-1})),\beta'[\alpha'(\beta'\beta(x_{1})),\ldots, \alpha'(\beta'\beta(x_{n-1})),\\&&\beta'(\alpha'\alpha(y_{k}))],\alpha'(\beta'^{2}\beta^{2}(y_{k+1})),\ldots,\alpha'(\beta'^{2}\beta^{2}(y_{n}))]\Big)\\[0.2cm]
&=&\displaystyle{\sum_{k=1}^{n}}  \varepsilon(X, Y_k) \Big(\{\beta'^{2}\beta^{2}(y_{1}),\ldots,\beta'^{2}\beta^{2}(y_{k-1}),\{\beta'\beta(x_{1}),\ldots, \beta'\beta(x_{n-1}),\\&&\alpha'\alpha(y_{k})\},\beta'^{2}\beta^{2}(y_{k+1}),\ldots,\beta'^{2}\beta^{2}(y_{n})\}\Big)\\[0.2cm]
&=&\displaystyle{\sum_{k=1}^{n}}  \varepsilon(X, Y_k) \Big(\{(\beta\beta')^{2}(y_{1}),\ldots,(\beta\beta')^{2}(y_{k-1}),\{\beta\beta'(x_{1}),\ldots, \beta\beta'(x_{n-1}),\\&&\alpha\alpha'(y_{k})\},(\beta\beta')^{2}(y_{k+1}),\ldots,(\beta\beta')^{2}(y_{n})\}\Big)
\end{array}$$
Second part is proved as follows:
\begin{eqnarray}
 f(\{x_1, \dots, x_n\})&=&f[\alpha'(x_1), \dots,\alpha'(x_{n-1}), \beta'(x_n)]\nonumber\\
&=&[f\alpha'(x_1), \dots,f\alpha'(x_{n-1}), f\beta'(x_n)]'\nonumber\\
&=&[\gamma'f(x_1), \dots,\gamma' f(x_{n-1}), \delta' f(x_n)]'\nonumber\\
&=&\{ f(x_1), \dots, f(x_n)\}'\nonumber
\end{eqnarray}
This completes the proof.
\end{proof}
\begin{corollary}
Let $(L, [\cdot, \dots, \cdot], \varepsilon, \alpha,\beta)$ be a multiplicative $n$-BiHom-Lie color algebra. Then, for any positive integer $k$,
$(L, [\cdot, \dots, \cdot]\circ(\alpha^{k}\otimes\dots\otimes\alpha^{k}\otimes\beta^{k}), \varepsilon, \alpha^{k+1},\beta^{k+1})$ is also an $n$-BiHom-Lie color algebra.
\end{corollary}
\begin{proof}
Apply Theorem \ref{yyt} with $\alpha'=\alpha^{k}$ and $\beta'=\beta^{k}$.
\end{proof}
\begin{theorem}\label{BakayokoSilvestrov:tp}
Let $(A, \cdot)$ be a commutative associative algebra and $(L, [\cdot, \dots, \cdot], \varepsilon, \alpha,\beta)$ be an $n$-BiHom-Lie color algebra. The tensor
product $A\otimes L=\sum_{\gamma\in \Gamma}(A\otimes L)_{\gamma}=\sum_{\gamma\in \Gamma}A\otimes L_\gamma$ with the bracket
$[a_1\otimes x_1, \dots, a_n\otimes x_n]'=a_1\dots a_n\otimes [x_1, \dots, x_n],$
the even linear map
$\alpha'(a\otimes x):=a\otimes \alpha(x)$ and $\beta'(a\otimes x):=a\otimes \beta(x)$
and the bicharacter
$\varepsilon(a+x, b+y)=\varepsilon(x, y), \forall a, b\in A, \forall x, y\in \mathcal{H}(L),$
is an $n$-BiHom-Lie color algebra.
\end{theorem}
\begin{proof}
For all $a\in A,~x\in L$
$$\beta'\alpha'(a\otimes x):=\beta'(a\otimes \alpha(x))=a\otimes \beta\alpha(x)=a\otimes \alpha\beta(x)=\alpha'\beta'(a\otimes x).$$
Then $\beta'\alpha'=\alpha'\beta'$.

With conditions in the statement, for all $a_{i},b_{i}\in A,~~x_{i},y_{i}\in L$, we have
$$
\begin{array}{lllll}
&&[\beta'(a_{1}\otimes x_{1}),\ldots,\beta'(a_{k}\otimes x_{k}),\beta'(a_{k+1}\otimes x_{k+1}),\dots,\beta'(a_{n-1}\otimes x_{n-1}),\alpha'(a_{n}\otimes x_{n})]'\\[0.2cm]
&=&[a_{1}\otimes\beta( x_{1}),\ldots,a_{k}\otimes\beta(x_{k}),a_{k+1}\otimes\beta(x_{k+1}),\dots,\ldots,a_{n-1}\otimes\beta( x_{n-1}),a_{n}\otimes\alpha( x_{n})]'\\[0.2cm]
&=&a_{1}\ldots a_{k} a_{k+1}\ldots a_{n}\otimes[\beta( x_{1}),\ldots,\beta(x_{k}),\beta(x_{k+1}),\dots,\beta( x_{n-1}),\alpha( x_{n})]\\[0.2cm]
&=&-a_{1}\ldots a_{k+1} a_{k}\ldots  a_{n}\varepsilon(x_{k},x_{k+1})[\beta( x_{1}),\ldots,\beta(x_{k+1}),\beta(x_{k}),\dots,\beta( x_{n-1}),\alpha( x_{n})]\\[0.2cm]
&=&-\varepsilon(x_{k},x_{k+1})[\beta'(a_{1}\otimes x_{1}),\ldots,\beta'(a_{k+1}\otimes x_{k+1}),\beta'(a_{k}\otimes x_{k}),\dots,\beta'(a_{n-1}\otimes x_{n-1}),\alpha'(a_{n}\otimes x_{n})]'\\[0.2cm]
&=&-\varepsilon(a_{k}+x_{k},a_{k+1}+x_{k+1})[\beta'(a_{1}\otimes x_{1}),\ldots,\beta'(a_{k+1}\otimes x_{k+1}),\beta'(a_{k}\otimes x_{k}),\dots,\beta'(a_{n-1}\otimes x_{n-1}),\\&&\alpha'(a_{n}\otimes x_{n})]'.\\[0.2cm]
\end{array}$$
In the same way,

$\begin{array}{llll}&&[\beta'(a_{1}\otimes x_{1}),\dots,\beta'(a_{n-1}\otimes x_{n-1}),\alpha'(a_{n}\otimes x_{n})]'\\[0.2cm]&=&
-\varepsilon(a_{n-1}+x_{n-1},a_{n}+x_{n})[\beta'(a_{1}\otimes x_{1}),\ldots,\beta'(a_{n}\otimes x_{n}),\alpha'(a_{n-1}\otimes x_{n-1})]'.\end{array}$\\

Finally, we prove  the $n$-BiHom $\varepsilon-$Jacobi identity of $A\otimes L$:
$$\begin{array}{llll}
&&[\beta'^{2}(a_{1}\otimes x_{1}), \dots, \beta'^{2}(a_{n-1}\otimes x_{n-1}), [\beta'(b_{1}\otimes y_{1}), \dots, \beta'(b_{n-1}\otimes y_{n-1}), \alpha'(b_{n}\otimes y_{n})]']' \nonumber\\[0.2cm]
&=&[[a_{1}\otimes\beta^{2}(x_{1}), \dots, a_{n-1}\otimes\beta^{2}(x_{n-1}), [b_{1}\otimes\beta(y_{1}), \dots, b_{n-1}\otimes\beta(y_{n-1}), \alpha(y_{n})]] \nonumber\\[0.2cm]
&=&a_{1}\dots a_{n-1}b_{1}\dots b_{n}[[\beta^{2}(x_{1}), \dots, \beta^{2}(x_{n-1}), [\beta(y_{1}), \dots, \beta(y_{n-1}), \alpha(y_{n})]] \nonumber\\[0.2cm]
&=&a_{1}\dots a_{n-1}b_{1}\dots b_{n}\displaystyle{\sum_{k=1}^{n}}  \varepsilon(X, Y_k) [\beta^{2}(y_{1}),\ldots,\beta^{2}(y_{k-1}),[\beta(x_{1}),\ldots, \beta(x_{n-1}),\alpha(y_{k})],\\[0.2cm]&&\beta^{2}(y_{k+1}),\ldots,\beta^{2}(y_{n})]\\[0.2cm]
&=&a_{1}\dots a_{n-1}b_{1}\dots b_{n}\displaystyle{\sum_{k=1}^{n}}  \varepsilon(X+A, (Y+B)_k) [\beta^{2}(y_{1}),\ldots,\beta^{2}(y_{k-1}),[\beta(x_{1}),\ldots, \beta(x_{n-1}),\\&&\alpha(y_{k})],\beta^{2}(y_{k+1}),\ldots,\beta^{2}(y_{n})]\\[0.2cm]
&=&\displaystyle{\sum_{k=1}^{n}}  \varepsilon(X+A, (Y+B)_k) [\beta'^{2}(b_{1}\otimes y_{1}),\ldots,\beta'^{2}(b_{k-1}\otimes y_{k-1}),[\beta'(a_{1}\otimes x_{1}),\ldots,\\[0.2cm]&& \beta'(a_{n-1}\otimes x_{n-1}),\alpha'(b_{k}\otimes y_{k})]',\beta'^{2}(b_{k+1}\otimes y_{k+1}),\ldots,\beta'^{2}(b_{n}\otimes y_{n})]',
  \end{array}$$
where $X+Y=\sum_{k=1}^{n-1}(x_k+a_{k}),~  (Y+B)_k=\sum_{l=1}^k(y_{l-1}+b_{l-1})$.

This completes the proof.
\end{proof}
\begin{proposition}
Given two $n$-BiHom-Lie color algebras $(L,[.,\dots,.],\varepsilon,\alpha,\beta)$ and $(L',[.,\dots,.]',\varepsilon,\alpha',\beta')$.
Then $(L\oplus L',[.,\dots,.]_{L\oplus L'},\alpha+\alpha',\beta+\beta'),$ is an $n$-BiHom-Lie color algebras,
where the $n-$linear map $[.,\dots,.]_{L\oplus L'}:\wedge^{n}(L\oplus L')\longrightarrow L\oplus L'$ is given by
$$[(u+v)_{1},\dots,(u+v)_{n}]_{L\oplus L'}=[u_{1},\dots,u_{n}]+[v_{1},\dots,v_{n}]',~\forall u_{i}\in L_{\gamma_{i}},v_{i}\in L'_{\gamma_{i}},~i=1,\dots,n;$$
and the two linear maps $\alpha+\alpha',\beta+\beta':L\oplus L'\longrightarrow L\oplus L',$ are given by $\forall u\in L,~v\in L'$
$$\begin{array}{llll}
(\alpha+\alpha')(u+v)&=&\alpha(u)+\alpha'(v),\\[0.2cm]
(\beta+\beta')(u+v)&=&\beta(u)+\beta'(v).
\end{array}$$
\end{proposition}
\begin{proof}
For any $(u+v)\in (L\oplus  L')_{\gamma_{i}}$  we have:
$$\begin{array}{llll}(\alpha+\alpha')\circ(\beta+\beta')(u+v)&=&(\alpha+\alpha')(\beta(u)+\beta'(v))\\[0.2cm]
&=&\alpha\circ\beta(u)+\alpha'\circ\beta'(v)\\[0.2cm]
&=&\beta\circ\alpha(u)+\beta'\circ\alpha'(v)\\[0.2cm]
&=&(\beta+\beta')(\alpha (u)+\alpha'(v))\\[0.2cm]
&=&(\beta+\beta')\circ(\alpha+\alpha')(u+v)\end{array}.$$
Then we have $(\alpha+\alpha')\circ(\beta+\beta')=(\beta+\beta')\circ(\alpha+\alpha')$.
Next, for any $L_{\gamma_{i}}\oplus L'_{\gamma_{i}}\ni u_{i}+v_{i}=(u+v)_{i}\in(L\oplus L')_{\gamma_{i}}$ we have 
$$
\begin{array}{lllll}
&&[(\beta+\beta')(u+ v)_{1},\ldots,(\beta+\beta')(u+ v)_{k},(\beta+\beta')(u+ v)_{k+1},\dots,\\&&(\beta+\beta')(u+ v)_{n-1},(\alpha+\alpha')(u+ v)_{n}]_{L\oplus L'}\\[0.2cm]
&=&[(\beta+\beta')(u_{1}+ v_{1}),\ldots,(\beta+\beta')(u_{k}+ v_{k}),(\beta+\beta')(u_{k+1}+ v_{k+1}),\dots,\\&&(\beta+\beta')(u_{n-1}+ v_{n-1}),(\alpha+\alpha')(u_{n}+ v_{n})]_{L\oplus L'}\\[0.2cm]
&=&[\beta(u_{1})+ \beta'(v_{1}),\ldots,\beta(u_{k})+ \beta'(v_{k}),\beta(u_{k+1})+ \beta'(v_{k+1}),\dots,\beta(u_{n-1})+\\&&\beta'(v_{n-1}),\alpha(u_{n})+ \alpha'(v_{n})]_{L\oplus L'}\\[0.2cm]
&=&[\beta(u_{1}),\ldots,\beta(u_{k}),\beta(u_{k+1}),\dots,\beta(u_{n-1}),\alpha(u_{n})]\\[0.2cm]
&&+[\beta'(v_{1}),\ldots,\beta'(v_{k}),\beta'(v_{k+1}),\dots,\beta'(v_{n-1}),\alpha'(v_{n})]'\\[0.2cm]
&=&-\varepsilon(u_{k},u_{k+1})[\beta(u_{1}),\ldots,\beta(u_{k+1}),\beta(u_{k}),\ldots,\beta(u_{n-1}),\alpha(u_{n})]\\[0.2cm]
&&-\varepsilon(v_{k},v_{k+1})[\beta'(v_{1}),\ldots,\beta'(v_{k+1}),\beta'(v_{k}),\dots,\beta'(v_{n-1}),\alpha'(v_{n})]'\\[0.2cm]
&=&-\varepsilon((u+ v)_{k},(u+ v)_{k+1})[(\beta+\beta')(u_{1}+ v_{1}),\ldots,(\beta+\beta')(u_{k+1}+ v_{k+1}),(\beta+\beta')(u_{k}+ v_{k}),\dots,\\&&(\beta+\beta')(u_{n-1}+ v_{n-1}),(\alpha+\alpha')(u_{n}+ v_{n})]_{L\oplus L'}\\[0.2cm]&=&-\varepsilon((u+ v)_{k},(u+ v)_{k+1})[(\beta+\beta')(u+ v)_{1},\ldots,(\beta+\beta')(u+ v)_{k+1},(\beta+\beta')(u+ v)_{k},\dots,\\&&(\beta+\beta')(u+ v)_{n-1},(\alpha+\alpha')(u+ v)_{n}]_{L\oplus L'}\\[0.2cm]
\end{array}$$
Similarly, we can get
$$
\begin{array}{lllll}
&&[(\beta+\beta')(u+ v)_{1},\ldots,(\beta+\beta')(u+ v)_{n-1},(\alpha+\alpha')(u+ v)_{n}]_{L\oplus L'}\\[0.2cm]
&=&-\varepsilon((u+ v)_{n-1},(u+ v)_{n})[(\beta+\beta')(u+ v)_{1},\ldots,(\beta+\beta')(u+ v)_{n},(\alpha+\alpha')(u+ v)_{n-1}]_{L\oplus L'}
\end{array}$$
We prove the $n$-BiHom-$\epsilon$-Jacobi identity,
$$\begin{array}{lllll}
&&[(\beta+\beta')^{2}(u+ v)_{1},\ldots,(\beta+\beta')^{2}(u+ v)_{n-1},[(\beta+\beta')(x+ y)_{1},\ldots,\\[0.2cm]
&&
(\beta+\beta')(x+ y)_{n-1},(\alpha+\alpha')(x+ y)_{n}]_{L\oplus L'}]_{L\oplus L'}\\[0.2cm]
&=&[(\beta+\beta')^{2}(u_{1}+v_{1}),\ldots,(\beta+\beta')^{2}(u_{n-1}+v_{n-1}),[(\beta+\beta')(x_{1}+y_{1}),\ldots,\\[0.2cm]
&&
(\beta+\beta')(x_{n-1}+y_{n-1}),(\alpha+\alpha')(x_{n}+y_{n})]_{L\oplus L'}]_{L\oplus L'}\\[0.2cm]
&=&[\beta^{2}(u_{1}),\ldots,\beta^{2}(u_{n-1}),[\beta(x_{1}),\ldots,
\beta(x_{n-1}),\alpha(x_{n})]\\[0.2cm]
&&+[\beta'^{2}(v_{1}),\ldots,\beta'^{2}(v_{n-1}),[\beta'(y_{1}),\ldots,
\beta'(y_{n-1}),\alpha'(y_{n})]'\\[0.2cm]
&=&\displaystyle{\sum_{k=1}^{n}}  \varepsilon(U, X_k) [\beta^{2}(x_{1}),\ldots,\beta^{2}(x_{k-1}),[\beta(u_{1}),\ldots, \beta(u_{n-1}),\alpha(x_{k})],\beta^{2}(x_{k+1}),\ldots,\beta^{2}(x_{n})]\\[0.2cm]
&&+\displaystyle{\sum_{k=1}^{n}}  \varepsilon(V, Y_k) [\beta'^{2}(y_{1}),\ldots,\beta'^{2}(y_{k-1}),[\beta'(v_{1}),\ldots, \beta'(v_{n-1}),\alpha'(y_{k})],\beta'^{2}(y_{k+1}),\ldots,\beta'^{2}(y_{n})]\\[0.2cm]
&=&\displaystyle{\sum_{k=1}^{n}}  \varepsilon(U+V, (X+Y)_k) [(\beta+\beta')^{2}(x_{1}+y_{1}),\ldots,(\beta+\beta')^{2}(x_{k-1}+y_{k-1}),\\&&[(\beta+\beta')(u_{1}+v_{1}),\ldots, (\beta+\beta')(u_{n-1}+v_{n-1}),(\alpha+\alpha')(x_{k}+y_{k})]_{L\oplus L'},\\&&(\beta+\beta')^{2}(x_{k+1}+y_{k+1}),\ldots,(\beta+\beta')^{2}(x_{n}+y_{n})]_{L\oplus L'}\\[0.2cm]
&=&\displaystyle{\sum_{k=1}^{n}}  \varepsilon(U+V, (X+Y)_k) [(\beta+\beta')^{2}(x+ y)_{1},\ldots,(\beta+\beta')^{2}(x+ y)_{k-1},\\&&[(\beta+\beta')(u+ v)_{1},\ldots, (\beta+\beta')(u+ v)_{n-1},(\alpha+\alpha')(x+ y)_{k}]_{L\oplus L'},\\&&(\beta+\beta')^{2}(x+ y)_{k+1},\ldots,(\beta+\beta')^{2}(x+ y)_{n}]_{L\oplus L'}\\[0.2cm]
\end{array}$$
where $U+V=\sum_{k=1}^{n-1}(u+v)_{k},~  (X+Y)_k=\sum_{l=1}^k(x+ y)_{l-1}$.

Then $(L\oplus L',[.,\ldots,.]_{L\oplus L'},\alpha+\alpha',\beta+\beta'),$ is an $n$-BiHom-Lie color algebras
\end{proof}
In the next definition, we introduce an element of the centroid (or semi-morphism) for $n$-BiHom-Lie color algebra.
\begin{definition}
A semi-morphism \index{Semi-morphism!n-Hom-Lie color algebra}  of an $n$-BiHom-Lie color algebra $(L, [\cdot, \dots, \cdot], \varepsilon, \alpha,\beta)$ is an even linear map $\gamma : L\rightarrow L$ such
that $\gamma\alpha=\alpha\gamma,~\gamma\beta=\beta\gamma$ and
$\gamma[x_1, \dots, x_n]=[x_1, \dots,\gamma(x_i), \dots, x_n].$
\end{definition}
\begin{theorem}\label{BakayokoSilvestrov:tch}
 Let $(L, [\cdot, \dots, \cdot], \varepsilon, \alpha,\beta)$ be an $n$-BiHom-Lie color algebra and $\gamma : L\rightarrow L$ a semi-morphism of $L$.
Define a new multiplication $\{\cdot, \dots, \cdot\}: L\times\dots\times L\rightarrow L$ by
$$\{x_1, \dots, x_n\}=[x_1, \dots, \gamma(x_i),\dots, x_n]$$
Then $(L, \{\cdot, \dots, \cdot\}, \varepsilon, \alpha,\beta)$ is also an $n$-BiHom-Lie color algebra.
\end{theorem}
\begin{proof}
For all $x_{i},y_{i}\in \mathcal{H}(L)$, the proof can be obtained as follows:

$$
\begin{array}{lllll}
&&\{\beta(x_{1}),\ldots,\beta(x_{i}),\ldots,\beta(x_{k}),\beta(x_{k+1}),\dots,\beta(x_{n-1}),\alpha(x_{n})\}\\[0.2cm]
&=&[\beta(x_{1}),\ldots,\gamma\beta(x_{i}),\ldots,\beta(x_{k}),\beta(x_{k+1}),\dots,\beta(x_{n-1}),\alpha(x_{n})]\\[0.2cm]
&=&-\varepsilon(x_{k},x_{k+1})[\beta(x_{1}),\ldots,\gamma\beta(x_{i}),\ldots,\beta(x_{k+1}),\beta(x_{k}),\dots,\beta(x_{n-1}),\alpha(x_{n})]\\[0.2cm]
&=&-\varepsilon(x_{k},x_{k+1})\{\beta(x_{1}),\ldots,\beta(x_{i}),\ldots,\beta(x_{k+1}),\beta(x_{k}),\dots,\beta(x_{n-1}),\alpha(x_{n})\}.
\end{array}$$
Similarly, we can get
$$\begin{array}{llllll}
&&\{\beta(x_{1}),\ldots,\beta(x_{i}),\ldots,\dots,\beta(x_{n-1}),\alpha(x_{n})\}\\
&=&-\varepsilon(x_{n-1},x_{n})\{\beta(x_{1}),\ldots,\beta(x_{i}),\ldots,\dots,\beta(x_{n}),\alpha(x_{n-1})\}.\end{array}$$
Finally, we have
 $$\begin{array}{lllllll} &&\{\beta^{2}(x_1), \dots, \beta^{2}(x_{n-1}), \{\beta(y_1), \dots, \beta(y_n)\}\}\nonumber\\[0.2cm]
&=&[\beta^{2}(x_1), \dots, \gamma\beta^{2}(x_i), \dots, \beta^{2}(x_{n-1}), [\beta(y_1), \dots, \gamma\beta(y_j), \dots, \beta(y_n)]]\nonumber\\[0.2cm]
&=&[\beta^{2}(x_1), \dots, \beta^{2}\gamma(x_i), \dots, \beta^{2}(x_{n-1}), [\beta(y_1), \dots, \beta\gamma(y_j), \dots, \beta(y_n)]]\nonumber\\[0.2cm]
&=&
\displaystyle{\sum_{k<j}}\varepsilon(X, Y_k)[\beta^{2}(y_1), \dots, \beta^{2}(y_{k-1}),[\beta(x_1), \dots, \gamma\beta(x_i), \dots, \beta(x_{n-1}), \alpha(y_k)], \nonumber \\[0.2cm]
&&\beta^{2}(y_{k+1}), \dots, \gamma\beta^{2}(y_j), \dots \beta^{2}(y_n)]
\nonumber\\[0.2cm]
&&+\varepsilon(X, Y_j)[\beta^{2}(y_1), \dots, \beta^{2}(y_{j-1}),
[\beta(x_1), \dots, \gamma\beta(x_i), \dots, \beta(x_{n-1}), \gamma\alpha(y_j)], \beta^{2}(y_{j+1}), \dots, \beta^{2}(y_n)]\nonumber\\[0.2cm]
&&
+\displaystyle{\sum_{k>j}}\varepsilon(X, Y_k)[\beta^{2}(y_1), \dots, \gamma\beta^{2}(y_{j}), \dots,\beta^{2}(y_{k-1}),
[\beta(x_1), \dots, \gamma\beta(x_i), \dots, \beta(x_{n-1}), \alpha(y_k)],  \nonumber \\[0.2cm]
&&\beta^{2}(y_{k+1}), \dots,\beta^{2}(y_n)]
\nonumber\\[0.2cm]
&=&
\displaystyle{\sum_{k<j}}\varepsilon(X, Y_k)[\beta^{2}(y_1), \dots, \beta^{2}(y_{k-1}),
\{\beta(x_1), \dots, \beta(x_i), \dots, \beta(x_{n-1}), \alpha(y_k)\},  \nonumber \\[0.2cm]
&&\beta^{2}(y_{k+1}), \dots, \gamma\beta^{2}(y_j), \dots \beta^{2}(y_n)]
\nonumber\\[0.2cm]
&&+\varepsilon(X, Y_j)[\beta^{2}(y_1), \dots, \beta^{2}(y_{j-1}),
\gamma(\{\beta(x_1), \dots, \beta(x_i), \dots, \beta(x_{n-1}), \alpha(y_j)\}), \beta^{2}(y_{j+1}), \dots, \beta^{2}(y_n)]\nonumber\\[0.2cm]
&&+\displaystyle{\sum_{k>j}}\varepsilon(X, Y_k)[\beta^{2}(y_1), \dots, \gamma\beta^{2}(y_{j}), \dots, \beta^{2}(y_{k-1}),
\{\beta(x_1), \dots, \beta(x_i), \dots, \beta(x_{n-1}), \alpha(y_k)\}, \nonumber\\[0.2cm]
&& \beta^{2}(y_{k+1}), \dots, \beta^{2}(y_n)]\nonumber\\[0.2cm]
&=&\displaystyle{\sum_{k<j}}\varepsilon(X, Y_k)\{\beta^{2}(y_1), \dots, \beta^{2}(y_{k-1}),
\{\beta(x_1), \dots, \beta(x_i), \dots, \beta(x_{n-1}), \alpha(y_k)\}, \nonumber \\[0.2cm]
&&  \beta^{2}(y_{k+1}), \dots, \beta^{2}(y_j), \dots \beta^{2}(y_n)\}\nonumber\\[0.2cm]
&&+\varepsilon(X, Y_j)\{\beta^{2}(y_1), \dots, \beta^{2}(y_{j-1}),
\{\beta(x_1), \dots, \beta(x_i), \dots, \beta(x_{n-1}), \alpha(y_j)\}, \beta^{2}(y_{j+1}), \dots, \beta^{2}(y_n)\}\nonumber\\[0.2cm]
&&+\displaystyle{\sum_{k>j}}\varepsilon(X, Y_k)\{\beta^{2}(y_1), \dots, \beta^{2}(y_{j}), \dots, \beta^{2}(y_{k-1}),
\{\beta(x_1), \dots, \beta(x_i), \dots, \beta(x_{n-1}), \alpha(y_k)\}, \nonumber \\[0.2cm]
&&  \beta^{2}(y_{k+1}), \dots, \beta^{2}(y_n)\}\nonumber.
 \end{array}$$
This completes the proof.
 \end{proof}
\begin{corollary}
 Let $(L, [\cdot, \dots, \cdot], \varepsilon)$ be an $n$-Lie color algebra and $\alpha : L\rightarrow L$ a semi-morphism of $L$.
Then $(L, \{\cdot, \dots, \cdot\}, \varepsilon)$ is  another $n$-Lie color algebra, with
$$\{x_1, \dots, x_n\}=[x_1, \dots, \alpha(x_i)\dots, x_n]$$
\end{corollary}
\begin{definition}
 Let $(L, [\cdot, \dots, \cdot], \varepsilon, \alpha,\beta)$ be an $n$-Hom-Lie color algebra. An averaging operator
 of an $n$-BiHom-Lie color algebra $L$ \index{Averaging operator!n-Hom-Lie color algebra} is an even linear map
 $\gamma : L\rightarrow L$ such that
\begin{enumerate}
 \item [(1)] $\gamma\alpha=\alpha\gamma,~~\gamma\beta=\beta\gamma$
\item [(2)] $\gamma[x_1, \dots, \gamma(x_i), \dots, x_n]=[x_1, \dots, \gamma(x_i), \dots, \gamma(x_j), \dots, x_n]$
\end{enumerate}
\end{definition}
\begin{theorem}
 Let $(L, [\cdot, \dots, \cdot], \varepsilon, \alpha,\beta)$ be an $n$-BiHom-Lie color algebra and $\gamma : L\rightarrow L$ an averaging operator of $L$.
Define a new multiplication $\{\cdot, \dots, \cdot\}: L\times\dots\times L\rightarrow L$ by
$$\{x_1, \dots, x_n\}=[x_1, \dots, \gamma(x_i)\dots, x_n]$$
Then $(L, \{\cdot, \dots, \cdot\}, \varepsilon, \alpha,\beta)$ is also an $n$-BiHom-Lie color algebra.
\end{theorem}
\begin{proof} It is similar to the one of Theorem \ref{BakayokoSilvestrov:tch}.
 \end{proof}
Taking $\alpha=\beta=id$, yields the following statement.
\begin{corollary}
 Let $(L, [\cdot, \dots, \cdot], \varepsilon)$ be an $n$-Lie color algebra and $\alpha : L\rightarrow L$ an averaging operator of $L$.
Then $(L, \{\cdot, \dots, \cdot\}, \varepsilon)$ is  another $n$-Lie color algebra, with
$$\{x_1, \dots, x_n\}=[x_1, \dots, \alpha(x_i)\dots, x_n]$$
\end{corollary}
Taking $\alpha=\gamma$, yields the following statement.
\begin{corollary}
 Let $(L, [\cdot, \dots, \cdot], \varepsilon, \alpha,\beta)$ be an $n$-BiHom-Lie color algebra and $\alpha : L\rightarrow L$ an averaging operator.
Define a new multiplication $\{\cdot, \dots, \cdot\}: L\times\dots\times L\rightarrow L$ by
$$\{x_1, \dots, x_n\}=[x_1, \dots, \alpha(x_i)\dots, x_n]$$
Then $(L, \{\cdot, \dots, \cdot\}, \varepsilon, \alpha,\beta)$ is also an $n$-BiHom-Lie color algebra.
\end{corollary}
Taking $\beta=\gamma$, yields the following statement.
\begin{corollary}
 Let $(L, [\cdot, \dots, \cdot], \varepsilon, \alpha,\beta)$ be an $n$-BiHom-Lie color algebra and $\beta : L\rightarrow L$ an averaging operator.
Define a new multiplication $\{\cdot, \dots, \cdot\}: L\times\dots\times L\rightarrow L$ by
$$\{x_1, \dots, x_n\}=[x_1, \dots, \beta(x_i)\dots, x_n]$$
Then $(L, \{\cdot, \dots, \cdot\}, \varepsilon, \alpha,\beta)$ is also an $n$-BiHom-Lie color algebra.
\end{corollary}
\begin{theorem}
 Let $(L, [\cdot, \dots, \cdot], \varepsilon, \alpha,\beta)$ be an $n$-BiHom-Lie color algebra and $\gamma : L\rightarrow L$ an averaging operator of $L$.
Then $(L, \{\cdot, \dots, \cdot\}, \varepsilon, \alpha,\beta)$ is  another $n$-BiHom-Lie color algebra, with
$$\{x_1, \dots, x_n\}=[x_1, \dots, \gamma(x_i),\dots,\gamma(x_j),\dots, x_n]$$
\end{theorem}
\begin{proof} The proof is obtained as follows:
 $$\begin{array}{llllllll}
&&\{\beta(x_{1}),\ldots,\beta(x_{i}),\ldots,\beta(x_{k}),\beta(x_{k+1}),\dots,\beta(x_{n-1}),\alpha(x_{n})\}\\[0.2cm]
&=&[\beta(x_{1}),\ldots,\gamma\beta(x_{i}),\ldots,\beta(x_{k}),\beta(x_{k+1}),\dots,\gamma\beta(x_{j}),\dots,\beta(x_{n-1}),\alpha(x_{n})]\\[0.2cm]
&=&-\varepsilon(x_{k},x_{k+1})[\beta(x_{1}),\ldots,\gamma\beta(x_{i}),\ldots,\beta(x_{k+1}),\beta(x_{k}),\dots,\gamma\beta(x_{j}),\dots,\beta(x_{n-1}),\alpha(x_{n})]\\[0.2cm]
&=&-\varepsilon(x_{k},x_{k+1})\{\beta(x_{1}),\ldots,\beta(x_{i}),\ldots,\beta(x_{k+1}),\beta(x_{k}),\dots,\beta(x_{n-1}),\alpha(x_{n})\}.
\end{array}$$

Similarly, we can get
\begin{eqnarray}
&&\{\beta(x_{1}),\ldots,\beta(x_{i}),\ldots,\dots,\beta(x_{n-1}),\alpha(x_{n})\}\nonumber\\
&=&-\varepsilon(x_{n-1},x_{n})\{\beta(x_{1}),\ldots,\beta(x_{i}),\ldots,\dots,\beta(x_{n}),\alpha(x_{n-1})\}\nonumber\\
&&[\beta^{2}(x_1), \dots, \gamma\beta^{2}(x_i), \dots, \gamma\beta^{2}(x_j), \dots, \beta^{2}(x_{n-1}),
[\beta(y_1), \dots, \gamma\beta(y_k), \dots, \gamma\beta(y_l), \dots, \alpha(y_n)]]\nonumber\\[0.2cm]
&=&\displaystyle{\sum_{m<k}}\varepsilon(X, Y_m)[\beta^{2}(y_1), \dots, \beta^{2}(y_{m-1}),
[\beta(x_1), \dots, \gamma\beta(x_i), \dots, \gamma\beta(x_j), \dots, \beta(x_{n-1}), \alpha(y_m)], \nonumber \\[0.2cm]
&&
\beta^{2}(y_{m+1}), \dots, \gamma\beta^{2}(y_k), \dots, \gamma\beta^{2}(y_l), \dots, \beta^{2}(y_n)]\nonumber\\[0.2cm]
&&+\varepsilon(X, Y_k)[\beta^{2}(y_1), \dots, \beta^{2}(y_{k-1}),
[\beta(x_1), \dots, \gamma\beta(x_i), \dots, \gamma\beta(x_j), \dots, \beta(x_{n-1}), \gamma\alpha(y_k)],\nonumber\\[0.2cm]
&&\beta^{2}(y_{k+1}), \dots, \gamma\beta^{2}(y_l), \dots, \beta^{2}(y_n)]\nonumber\\[0.2cm]
&&+\displaystyle{\sum_{k<m<l}}\varepsilon(X, Y_m)[\beta^{2}(y_1), \dots, \gamma\beta^{2}(y_{k}), \dots, \beta^{2}(y_{m-1}),
\nonumber \\[0.2cm]
&&
[\beta(x_1), \dots, \gamma\beta(x_i), \dots, \gamma\beta(x_j), \dots, \beta (x_{n-1}), \alpha(y_m)], \beta^{2}(y_{m+1}), \dots, \gamma\beta^{2}(y_{l}), \dots, \gamma\beta^{2}(y_n)]\nonumber\\[0.2cm]
&&+\varepsilon(X, Y_l)[\beta^{2}(y_1), \dots, \gamma\beta^{2}(y_{k}),\dots, \beta^{2}(y_{l-1}),
\nonumber\\[0.2cm]
&&
[\beta(x_1), \dots, \gamma\beta(x_i), \dots, \gamma\beta(x_j), \dots, \beta(x_{n-1}), \gamma\alpha(y_l)],\beta^{2}(y_{l+1}), \dots, \beta^{2}(y_n)]\nonumber\\[0.2cm]
&&+\displaystyle{\sum_{m>l}}\varepsilon(X, Y_m)[\beta^{2}(y_1), \dots, \gamma\beta^{2}(y_{k}), \dots, \gamma\beta^{2}(y_{l}), \dots, \beta^{2}(y_{m-1}),
\nonumber\\[0.2cm]
&&
[\beta(x_1), \dots, \gamma\beta(x_i), \dots, \gamma\beta(x_j), \dots, \beta(x_{n-1}), \alpha(y_m)],\beta^{2}(y_{m+1}), \dots,  \beta^{2}(y_n)]\nonumber\\[0.2cm]
&=&\displaystyle{\sum_{m<k}}\varepsilon(X, Y_m)\{\beta^{2}(y_1), \dots, \beta^{2}(y_{m-1}),
\{\beta(x_1), \dots, \beta(x_i), \dots, \beta(x_j), \dots, \beta(x_{n-1}), \alpha(y_m)\},
\nonumber\\[0.2cm]
&&
\beta^{2}(y_{m+1}),\dots, \beta^{2}(y_k), \dots, \beta^{2}(y_l), \dots, \beta^{2}(y_n)\}\nonumber\\[0.2cm]
&&+\varepsilon(X, Y_k)[\beta^{2}(y_1), \dots, \beta^{2}(y_{k-1}),
\gamma([\beta(x_1), \dots, \gamma\beta(x_i), \dots, \gamma\beta(x_j), \dots, \beta(x_{n-1}), \alpha(y_k)]),\nonumber\\[0.2cm]
&& \beta^{2}(y_{k+1}), \dots, \gamma\beta^{2}(y_l), \dots, \beta^{2}(y_n)]\nonumber\\[0.2cm]
&&+\displaystyle{\sum_{k<m<l}}\varepsilon(X, Y_m)\{\beta^{2}(y_1), \dots, \beta^{2}(y_{k}), \dots, \beta^{2}(y_{m-1}),
\nonumber\\[0.2cm]
&&
\{\beta(x_1), \dots, \beta(x_i), \dots, \beta(x_j), \dots, \beta(x_{n-1}), \alpha(y_m)\},\beta^{2}(y_{m+1}), \dots, \beta^{2}(y_{l}), \dots, \beta^{2}(y_n)]\nonumber\\[0.2cm]
&&+\varepsilon(X, Y_l)[\beta^{2}(y_1), \dots, \gamma\beta^{2}(y_{k}),\dots, \beta^{2}(y_{l-1}),
\nonumber\\[0.2cm]
&&
\gamma([\beta(x_1), \dots, \gamma\beta(x_i), \dots, \gamma\beta(x_j), \dots, \beta(x_{n-1}), \alpha(y_l]), \beta^{2}(y_{l+1}), \dots, \beta^{2}(y_n)]\nonumber\\[0.2cm]
&&+\displaystyle{\sum_{m>l}}\varepsilon(X, Y_m)\{\beta^{2}(y_1), \dots, \beta^{2}(y_{k}), \dots, \beta^{2}(y_{l}), \dots, \beta^{2}(y_{m-1}),
\nonumber\\[0.2cm]
&&
\{\beta(x_1), \dots, \beta(x_i), \dots, \beta(x_j), \dots, \beta(x_{n-1}), \alpha(y_m)\},\beta^{2}(y_{m+1}), \dots,  \beta^{2}(y_n)\}\nonumber\\
 \end{eqnarray}
\newpage
\begin{eqnarray}
&=&\displaystyle{\sum_{m<k}}\varepsilon(X, Y_m)\{\beta^{2}(y_1), \dots, \beta^{2}(y_{m-1}),
\{\beta(x_1), \dots, \beta(x_i), \dots, \beta(x_j), \dots, \beta(x_{n-1}), \alpha(y_m)\},
\nonumber\\[0.2cm]
&&
\beta^{2}(y_{m+1}), \dots,\beta^{2}(y_k), \dots, \beta^{2}(y_l), \dots,\beta^{2}(y_n)\}\nonumber\\[0.2cm]
&&+\varepsilon(X, Y_k)\{\beta^{2}(y_1), \dots, \beta^{2}(y_{k-1}),
\{\beta(x_1), \dots, \beta(x_i), \dots, \beta(x_j), \dots, \beta(x_{n-1}), \alpha(y_k)\},\nonumber\\[0.2cm]
&&\beta^{2}(y_{k+1}), \dots, \beta^{2}(y_l), \dots, \beta^{2}(y_n)\}\nonumber\\[0.2cm]
&& +\displaystyle{\sum_{k<m<l}}\varepsilon(X, Y_m)\{\beta^{2}(y_1), \dots, \beta^{2}(y_{k}), \dots, \beta^{2}(y_{m-1}),
\nonumber\\[0.2cm]
&&
\{\beta(x_1), \dots, \beta(x_i), \dots, \beta(x_j), \dots, \beta(x_{n-1}), \alpha(y_m)\},\beta^{2}(y_{m+1}), \dots, \beta^{2}(y_{l}), \dots, \beta^{2}(y_n)\}\nonumber\\[0.2cm]
&&+\varepsilon(X, Y_l)\{\beta^{2}(y_1), \dots, \beta^{2}(y_{k}),\dots, \beta^{2}(y_{l-1}),
\{\beta(x_1), \dots, \beta(x_i), \dots, \beta(x_j), \dots, \beta(x_{n-1}), \alpha(y_l)\}),\nonumber\\[0.2cm]
&&\beta^{2}(y_{l+1}), \dots, \beta^{2}(y_n)\}\nonumber\\[0.2cm]
&&+\displaystyle{\sum_{m>l}}\varepsilon(X, Y_m)\{\beta^{2}(y_1), \dots, \beta^{2}(y_{k}), \dots, \beta^{2}(y_{l}), \dots, \beta^{2}(y_{m-1}),
\nonumber\\[0.2cm]
&&
\{\beta(x_1), \dots, \beta(x_i), \dots, \beta(x_j), \dots, \beta(x_{n-1}), \alpha(y_m)\},\beta^{2}(y_{m+1}), \dots,  \beta^{2}(y_n)\}\nonumber.
 \end{eqnarray}
This finishes the proof.
 \end{proof}
 \begin{proposition}
A homomorphism $f:(L,[\cdot,\ldots,\cdot],\alpha,\beta)\longrightarrow(L',[\cdot,\ldots,\cdot]',\alpha',\beta')$ is a morphism of $n$-BiHom-Lie color algebras if and only if the graph
$L\oplus L'$ is an $n$-BiHom-subalgebra of $(L\oplus L',[\cdot,\ldots,\cdot]_{L\oplus L'},\alpha+\alpha',\beta+\beta')$.
\end{proposition}
\begin{proof}
Let $f:(L,[\cdot,\ldots,\cdot],\alpha,\beta)\longrightarrow(L',[\cdot,\ldots,\cdot]',\alpha',\beta')$ is a morphism of $n$-BiHom-Lie color algebras, then for any $x,x_{1},\ldots,x_{n}\in L$, we have
$$[x_{1}+f(x_{1}),\ldots,x_{n}+f(x_{n})]_{L \oplus L'}=[x_{1},\ldots,x_{n}]+[f(x_{1}),\ldots,f(x_{n})]'=[x_{1},\ldots,x_{n}]+f([x_{1},\ldots,x_{n}]).$$
Thus the graph $\phi_{f}$ is closed under the bracket operation $[\cdot,\ldots,\cdot]_{L\oplus L'}$. Furthermore, we have
$$(\alpha+\alpha')(x+f(x))=\alpha(x)+\alpha'\circ f(x)=\alpha(x)+f\circ\alpha(x),$$
which implies that $$(\alpha+\alpha')(\phi_{f})\subset\phi_{f}.$$
Similarly,$$(\beta+\beta')(\phi_{f})\subset\phi_{f}.$$
Thus $\phi_{f}$ is an $n$-BiHom-Lie color subalgebras of $(L\oplus L',[\cdot,\ldots,\cdot]_{L\oplus L'},\alpha+\alpha',\beta+\beta')$.

Conversely, if the graph $\phi_{f}$ is an $n$-BiHom-Lie color subalgebras of $(L\oplus L',[\cdot,\ldots,\cdot]_{L\oplus L'},\alpha+\alpha',\beta+\beta')$, we have
$$[x_{1}+f(x_{1}),\ldots,x_{n}+f(x_{n})]_{L \oplus L'}=[x_{1},\ldots,x_{n}]+[f(x_{1}),\ldots,f(x_{n})]'\in\phi_{f}$$
which implies that $$[f(x_{1}),\ldots,f(x_{n})]'=f([x_{1},\ldots,x_{n}]).$$
Furthermore, $(\alpha+\alpha')(\phi_{f})\subset\phi_{f}$ yields that
$$(\alpha+\alpha')(x+f(x))=\alpha(x)+\alpha'\circ f(x)\in\phi_{f},$$
which is equivalent to the condition $\alpha'\circ f(x)=f\circ \alpha(x)$, i.e, $\alpha'\circ f=f\circ\alpha.$ Similarly, $\beta'\circ f=f\circ\beta.$ Therefore, $f$ is a morphism of
$n$-BiHom-Lie color algebras.
\end{proof}

\section{BiHom-modules over $n$-BiHom-Lie color algebras}
\label{BakayokoSilvestrov:sec:modulesnhomcoloralg}
In this section we consider BiHom-modules over $n$-BiHom-Lie color algebras.

\begin{definition}
Let $\Gamma$ be an abelian group. A BiHom-module \index{Hom-module} is a pair $(M,\alpha_M,\beta_{M})$ in which $M$ is a $\Gamma$-graded linear space and $\alpha_M,\beta_{M}: M\longrightarrow M$ are an even linear maps such that $\alpha_{M}\beta_{M}=\beta_{M}\alpha_{M}$.
\end{definition}

\begin{definition}
  Let $(L, [\cdot, \dots, \cdot], \varepsilon, \alpha,\beta)$ be an $n$-BiHom-Lie color algebra and $(M, \alpha_M,\beta_{M})$ a BiHom-module. The BiHom-module $(M, \alpha_M,\beta_{M})$
 is called an  $n$-BiHom-Lie module \index{Hom-module!n-Hom-Lie} over $L$ if there are  $n$ polylinear maps:
$$\omega_i : L\otimes \dots L\otimes\underbrace{M}_{i}\otimes L\otimes\dots\otimes L\rightarrow M, \quad i=1, 2, \dots, n$$
such that, for any $x_i, y_i\in\mathcal{H}(L)$ and $m\in\mathcal{H}(M)$,
\begin{enumerate}
 \item[a)]$\omega_i(\beta(x_1), \dots, \beta(x_{i-1}), \beta_{M}(m), \beta(x_{i+1}), \dots, \alpha(x_n))$ is a $\varepsilon$-skew-symmetric by all $x$-type arguments.
 \item[b)]$\omega_i(\beta(x_1), \dots, \beta(x_{i-1}), \beta_{M}(m), \beta(x_{i+1}), \dots, \alpha(x_n))=-\varepsilon(m, x_{i+1})\omega_{i+1}(\beta(x_1), \dots,\\[0.2cm] \beta(x_{i-1}), \beta(x_{i+1}), \beta_{M}(m),\dots, \alpha(x_n))$ ;\\[0.3cm]
     for $i=1, 2, \dots, n-1$.\\[0.3cm]
$\omega_{n-1}(\beta(x_1), \dots,\beta(x_{n-2}), \beta_{M}(m),\alpha(x_n))=-\varepsilon(m, x_{n})\omega_{n}(\beta(x_1),\dots, \beta(x_{n-2}),\beta(x_n), \alpha_{M}(m))$; \\[0.3cm]
 \item[c)]
$\omega_n(\beta^{2}(x_1), \dots, \beta^{2}(x_{n-1}), \omega_n(\beta(y_1), \dots, \beta(y_{n-1}), \alpha_{M}(m)))=$
\begin{eqnarray}
&=&\sum_{i=1}^{n-1}\varepsilon(X, Y_i)\omega_n(\beta^{2}(y_1), \dots, \beta^{2}(y_{i-1}),
[\beta(x_1), \dots, \beta(x_{n-1}), \alpha(y_i)], \beta^{2}(y_{i+1}), \dots, \beta^{2}_M(m))\nonumber\\
&&+\varepsilon(X, Y_n)\omega_n(\beta^{2}(y_1), \dots, \beta^{2}(y_{n-1}), \omega_n(\beta(x_1), \dots, \beta(x_{n-1}), \alpha_{M}(m))),\nonumber
\end{eqnarray}
where $x_i, y_j\in\mathcal{H}(L)$, $X=\sum_{i=1}^{n-1}x_i$,  $Y_i=\sum_{j=1}^iy_{j-1}, y_0=e$ and $m\in\mathcal{H}(M)$.
\item[d)]
$\omega_{n-1}(\beta^{2}(x_1), \dots, \beta^{2}(x_{n-2}), \beta^{2}_M(m), [\beta(y_1), \dots, \alpha(y_{n})])=$
\begin{eqnarray}
&=&\sum_{i=1}^{n}\varepsilon(X, Y_i)\omega_i(\beta^{2}(y_1), \dots, \beta^{2}(y_{i-1}),
\omega_{n-1}(\beta(x_1), \dots, \beta(x_{n-2}), \beta_{M}(m), \alpha(y_i)),
\nonumber \\
&&
\beta^{2}(y_{i+1}), \dots, \beta^{2}(y_n))\nonumber,
\end{eqnarray}
where $x_i, y_j\in\mathcal{H}(L)$, $X=\sum_{i=1}^{n-2}x_i+m$,  $Y_i=\sum_{j=1}^iy_{j-1}, y_0=e$ and $m\in\mathcal{H}(M)$.
\end{enumerate}
\end{definition}

\begin{example}
 Any $n$-BiHom-Lie color algebra $(L, [\cdot, \dots, \cdot], \varepsilon, \alpha,\beta)$ is an $n$-BiHom-Lie module over itself by taking $M=L$, $\alpha_M=\alpha,~\beta_{M}=\beta$
and $\omega_i(\cdot,\dots, \cdot)=[\cdot, \dots, \cdot]$.
\end{example}

\begin{theorem}
  Let $(L, [\cdot, \dots, \cdot], \varepsilon, \alpha,\beta)$ be an $n$-BiHom-Lie color algebra, and let $(M, \alpha_M,\beta_{M}, \omega_i)$ be an $n$-BiHom-Lie color module
and $\gamma : L\rightarrow L$ be an endomorphism. Define
$$\tilde\omega_i=\omega_{i}(\gamma, \dots, \gamma, \underbrace{id}_i, \gamma, \dots, \gamma),~ i=1, 2, \dots, n.$$
Then $(M, \alpha_M,\beta_{M}, \tilde\omega_i)$ is an $n$-BiHom-Lie color module.
\end{theorem}
\begin{proof}
 The item $a)$ is obvious, we prove the item $b)$.

$$\begin{array}{lllll}&&\tilde\omega_i(\beta(x_1), \dots, \beta(x_{i-1}), \beta_{M}(m), \beta(x_{i+1}), \dots, \alpha(x_n))\\[0.2cm]&=&\omega_i(\gamma\beta(x_1), \dots, \gamma\beta(x_{i-1}), \beta_{M}(m),\gamma \beta(x_{i+1}), \dots, \gamma\alpha(x_n))\\[0.2cm]&=&-\varepsilon(m, x_{i+1})\omega_{i+1}(\gamma\beta(x_1), \dots,\gamma \beta(x_{i-1}), \gamma\beta(x_{i+1}), \beta_{M}(m),\dots,\gamma \alpha(x_n)) ;\\[0.2cm]
&=&-\varepsilon(m, x_{i+1})\tilde\omega_{i+1}(\beta(x_1), \dots, \beta(x_{i-1}), \beta(x_{i+1}), \beta_{M}(m), \dots, \alpha(x_n)).
\end{array}$$
Similarly, we have
$$\widetilde{\omega}_{n-1}(\beta(x_1), \dots,\beta(x_{n-2}), \beta_{M}(m),\alpha(x_n))=-\varepsilon(m, x_{n})\widetilde{\omega}_{n}(\beta(x_1),\dots, \beta(x_{n-2}),\beta(x_n), \alpha_{M}(m))$$.

So we only prove $c)$, item $d)$ being proved similarly.
$$\begin{array}{llll}
&&\tilde\omega_n(\beta^{2}(x_1), \dots, \beta^{2}(x_{n-1}), \tilde\omega_n(\beta(y_1), \dots, \beta(y_{n-1}), \alpha_{M}(m)))=\nonumber\\[0.2cm]
&=&\omega_n(\gamma\beta^{2}(x_1), \dots, \gamma\beta^{2}(x_{n-1}), \omega_n(\gamma\beta(y_1), \dots, \gamma\beta(y_{n-1}), \alpha_{M}(m)))\nonumber\\[0.2cm]
&=&\omega_n(\beta^{2}\gamma(x_1), \dots, \beta^{2}\gamma(x_{n-1}), \omega_n(\beta\gamma(y_1), \dots, \beta\gamma(y_{n-1}), \alpha_{M}(m)))\nonumber\\[0.2cm]
&=&\displaystyle{\sum_{i=1}^{n-1}}\varepsilon(X, Y_i)\omega_n(\beta^{2}\gamma(y_1), \dots, \beta^{2}\gamma(y_{i-1}),
[\beta\gamma(x_1), \dots, \beta\gamma(x_{n-1}), \beta\gamma(y_i)], \nonumber\\[0.2cm]
&&\beta^{2}\gamma(y_{i+1}), \dots, \beta_{M}^{2}(m)) \nonumber\\[0.2cm]
&&  +\varepsilon(X, Y_n)\omega_n(\beta^{2}\gamma(y_1), \dots, \beta^{2}\gamma(y_{n-1}), \omega_n(\beta\gamma(x_1), \dots, \beta\gamma(x_{n-1}), \alpha_{M}(m)))\nonumber\\[0.2cm]
&=&\displaystyle{\sum_{i=1}^{n-1}}\varepsilon(X, Y_i)\omega_n(\gamma\beta^{2}(y_1), \dots, \gamma\beta^{2}(y_{i-1}),
\gamma([\beta(x_1), \dots, \beta(x_{n-1}), \alpha(y_i)]),
\nonumber\\[0.2cm]
&&
\gamma\beta^{2}(y_{i+1}), \dots, \beta^{2}_M(m))\nonumber\\[0.2cm]
&&+\varepsilon(X, Y_n)\omega_n(\gamma\beta^{2}(y_1), \dots, \gamma\beta^{2}(y_{n-1}), \omega_n(\gamma\beta(x_1), \dots, \gamma\beta(x_{n-1}), \alpha_{M}(m))),\nonumber\\[0.2cm]
&=&\displaystyle{\sum_{i=1}^{n-1}}\varepsilon(X, Y_i)\omega_n(\gamma\otimes\dots\otimes\gamma\otimes id)(\beta^{2}(y_1), \dots, \beta^{2}(y_{i-1}),
[\beta(x_1), \dots, \beta(x_{n-1}), \alpha(y_i)],
\nonumber\\[0.2cm]
&&
\beta^{2}(y_{i+1}), \dots, \beta^{2}_M(m))\nonumber\\[0.2cm]
&& +\varepsilon(X, Y_n)\omega_n(\gamma\otimes\dots\otimes\gamma\otimes id)\Big(\beta^{2}(y_{1})\otimes \dots\otimes \beta^{2}(y_{n-1})\otimes
 \omega_n(\beta(x_{1})\otimes\dots\otimes\beta(x_{n-1})\otimes \alpha_{M}(m))\Big)
 \nonumber\\[0.2cm]
&=&\displaystyle{\sum_{i=1}^{n-1}}\varepsilon(X, Y_i)\tilde\omega_n(\beta^{2}(y_1), \dots, \beta^{2}(y_{i-1}),
[\beta(x_1), \dots, \beta(x_{n-1}), \alpha(y_i)], \beta^{2}(y_{i+1}), \dots, \beta^{2}_M(m))\nonumber\\[0.2cm]
&&+\varepsilon(X, Y_n)\tilde\omega_n(\beta^{2}(y_1), \dots,\beta^{2}(y_{n-1}),
 \tilde\omega_n(\beta(x_1), \dots, \beta(x_{n-1}), \alpha_{M}(m)))\nonumber.
\end{array}$$
This ends the proof.
 \end{proof}
\begin{corollary}
 Let $(L, [\cdot, \dots, \cdot], \varepsilon, \alpha,\beta)$ be an $n$-BiHom-Lie color algebra and $\gamma : L\rightarrow L$ be an endomorphism.
Then $(L, \{\cdot, \dots, \cdot\}_i, \alpha,\beta)$, with
$$\{\cdot, \dots, \cdot\}_i=[\gamma, \dots, \gamma,\underbrace{id}_i, \gamma, \dots, \gamma], \quad i=1, 2, \dots, n,$$
is an $n$-BiHom-Lie color module.
\end{corollary}

\begin{corollary}
 Let $(L, [\cdot, \dots, \cdot], \varepsilon, \alpha,\beta)$ be a multiplicative $n$-BiHom-Lie color algebra.
Then, for any $k,l\geq 1$, $(L, \{\cdot, \dots, \cdot\}_i^{k,l}, \alpha,\beta)$ is an $n$-BiHom-Lie color module, with
$$\{\cdot, \dots, \cdot\}_i^{k,l}=[\alpha^k\beta^{l}, \dots, \alpha^k\beta^{l},\underbrace{id}_i, \alpha^k\beta^{l}, \dots, \alpha^k\beta^{l}], \quad i=1, 2, \dots, n.$$
\end{corollary}
We end this section by giving some results for trivial gradation i.e. $\Gamma=\{e\}$.

\begin{proposition}
 Let $(M, \alpha_M, \beta_{M},\omega_i)$ be a module over the $n$-BiHom-Lie algebra \\ $(L, [\cdot, \dots, \cdot], \alpha_L,\beta_{L})$. Consider the direct sum of linear spaces
$A=L\oplus M$. Let' us define on $A$ the bracket
\begin{itemize}
 \item $\{x_1, \dots, x_n\}=[x_1, \dots, x_n]$,
\item $\{x_1, \dots, x_{i-1}, m, x_{i+1}, \dots, x_n\}=\omega_i(x_1, \dots, x_{i-1}, m, x_{i+1}, \dots, x_n)$,
\item $\{x_1, \dots, x_{i}, \dots, x_{j}, \dots, x_n\}=0$, whenever $x_i, x_j\in M$.
\end{itemize}
Then $(A,  \{\cdot,\dots,\cdot\},\alpha_A=\alpha_L+\alpha_M, \beta_A=\beta_L+\beta_M)$ is an $n$-BiHom-Lie algebra.
\end{proposition}

\begin{proposition}
 Let $(M, \alpha_M, \beta_{M},\omega_i)$ be a module over the $n$-BiHom-Lie algebra \\ $(L, [\cdot, \dots, \cdot], \alpha_L,\beta_{L})$. Consider the direct sum of linear spaces
$A=L\oplus M$. Let us define on $A$ the bracket
$$
\{x_1+m_1, \dots, x_n+m_n\}=[x_1, \dots, x_n]+\sum_{i=1}^{n}\omega_{i}(x_{1},\dots,x_{i-1},m_{i},x_{i+1},\dots,x_n)$$

Then $(A,  \{\cdot,\dots,\cdot\},\alpha_A=\alpha_L+\alpha_M, \beta_A=\beta_L+\beta_M)$ is an $n$-BiHom-Lie algebra.
\end{proposition}

\begin{proposition}
 Let $(M_1, \alpha_M^1,\beta_{M}^{1}, \omega_i^1)$ and $(M_2, \alpha_M^2,\beta_{M}^{2}, \omega_i^2)$ be two modules over the $n$-BiHom-Lie algebra $(L, [\cdot, \dots, \cdot], \alpha,\beta)$.
Then  $(M, \alpha_M,\beta_{M}, \omega_i)$ is an $n$-BiHom-Lie module with
$$M=M_1\oplus M_2,\quad \alpha_M=\alpha_M^1\oplus \alpha_M^2,~\quad \beta_M=\beta_M^1\oplus \beta_M^2\quad\mbox{and}\quad\omega_i=\omega_i^1\oplus\omega_i^2.$$
\end{proposition}
\begin{proof}
We prove the item $b)$.
$$\begin{array}{lllll}&&\omega_i(\beta(x_1), \dots, \beta(x_{i-1}), \beta_{M}(m), \beta(x_{i+1}), \dots, \alpha(x_n))\\[0.2cm]&=&(\omega_i^{1}\oplus\omega_{i}^{2})(\beta(x_1), \dots, \beta(x_{i-1}), (\beta_{M}^{1}\oplus\beta_{M}^{2})(m_{1}+m_{2}), \beta(x_{i+1}), \dots, \alpha(x_n))\\[0.2cm]
&=&\omega_i^{1}(\beta(x_1), \dots, \beta(x_{i-1}), \beta_{M}^{1}(m_{1}), \beta(x_{i+1}), \dots, \alpha(x_n))+\omega_{i}^{2}(\beta(x_1), \dots, \beta(x_{i-1}), \beta_{M}^{2}(m_{2}), \beta(x_{i+1}), \dots, \alpha(x_n))\\[0.2cm]
&=&-\varepsilon(m_1, x_{i+1})\omega_{i+1}^{1}(\beta(x_1), \dots, \beta(x_{i-1}), \beta(x_{i+1}), \beta_{M}^{1}(m_{1}),\dots, \alpha(x_n))\\[0.2cm]
&&-\varepsilon(m_2, x_{i+1})\omega_{i+1}^{2}(\beta(x_1), \dots, \beta(x_{i-1}), \beta(x_{i+1}), \beta_{M}^{2}(m_{2}),\dots, \alpha(x_n))
\\[0.2cm]&=&\varepsilon(m, x_{i+1})(\omega_{i+1}^{1}\oplus\omega_{i+1}^{2})(\beta(x_1), \dots, \beta(x_{i-1}), \beta(x_{i+1}), (\beta_{M}^{1}\oplus\beta_{M}^{2})(m),\dots, \alpha(x_n))\\[0.2cm]
&=&-\varepsilon(m, x_{i+1})\omega_{i+1}(\beta(x_1), \dots, \beta(x_{i-1}), \beta(x_{i+1}),\beta_{M}(m),  \dots, \alpha(x_n))
\end{array}$$
Similarly, we have
$$\omega_{n-1}(\beta(x_1), \dots,\beta(x_{n-2}), \beta_{M}(m),\alpha(x_n))=-\varepsilon(m, x_{n})\omega_{n}(\beta(x_1),\dots, \beta(x_{n-2}),\beta(x_n), \alpha_{M}(m))$$.

So we only prove $c)$, item $d)$ being proved similarly.
$$\begin{array}{llll}
&&\omega_n(\beta^{2}(x_1), \dots, \beta^{2}(x_{n-1}), \omega_n(\beta(y_1), \dots, \beta(y_{n-1}), \alpha_{M}(m)))\nonumber\\[0.2cm]
&=&(\omega_n^{1}\oplus\omega_{n}^{2})(\beta^{2}(x_1), \dots, \beta^{2}(x_{n-1}), (\omega_n^{1}\oplus\omega_{n}^{2})(\beta(y_1), \dots, \beta(y_{n-1}), (\alpha_{M}^{1}\oplus\alpha_{M}^{2})(m_{1}+m_{2}))\nonumber\\[0.2cm]
&=&\omega_n^{1}(\beta^{2}(x_1), \dots, \beta^{2}(x_{n-1}), \omega_n^{1}(\beta(y_1), \dots, \beta(y_{n-1}), \alpha_{M}^{1}(m_{1}))\nonumber\\[0.2cm]
&=&\omega_n^{2}(\beta^{2}(x_1), \dots, \beta^{2}(x_{n-1}), \omega_n^{2}(\beta(y_1), \dots, \beta(y_{n-1}), \alpha_{M}^{2}(m_{2}))\nonumber\\[0.2cm]
&=&\displaystyle{\sum_{i=1}^{n-1}}\varepsilon(X, Y_i)\omega_n^{1}(\beta^{2}(y_1), \dots, \beta^{2}(y_{i-1}),
[\beta(x_1), \dots, \beta(x_{n-1}), \beta(y_i)], \nonumber\\[0.2cm]
&&\beta^{2}(y_{i+1}), \dots, (\beta_{M}^{1})^{2}(m_{1})) +\nonumber\\[0.2cm]
&&  +\varepsilon(X, Y_n)\omega_n^{1}(\beta^{2}(y_1), \dots, \beta^{2}(y_{n-1}), \omega_n^{1}(\beta(x_1), \dots, \beta(x_{n-1}), \alpha_{M}^{1}(m_{1})))\nonumber\\[0.2cm]
&&+\displaystyle{\sum_{i=1}^{n-1}}\varepsilon(X, Y_i)\omega_n^{2}(\beta^{2}(y_1), \dots, \beta^{2}(y_{i-1}),
[\beta(x_1), \dots, \beta(x_{n-1}), \beta(y_i)], \nonumber\\[0.2cm]
&&\beta^{2}(y_{i+1}), \dots, (\beta_{M}^{2})^{2}(m_{2})) +\nonumber\\[0.2cm]
&&  +\varepsilon(X, Y_n)\omega_n^{2}(\beta^{2}(y_1), \dots, \beta^{2}(y_{n-1}), \omega_n^{2}(\beta(x_1), \dots, \beta(x_{n-1}), \alpha_{M}^{2}(m_{2})))\nonumber\\[0.2cm]
&=&\displaystyle{\sum_{i=1}^{n-1}}\varepsilon(X, Y_i)(\omega_n^{1}\oplus \omega_n^{2}) (\beta^{2}(y_1), \dots, \beta^{2}(y_{i-1}),
[\beta(x_1), \dots, \beta(x_{n-1}), \beta(y_i)], \nonumber\\[0.2cm]
&&\beta^{2}(y_{i+1}), \dots, ((\beta_{M}^{1})^{2}\oplus (\beta_{M}^{1})^{2})(m)) +\nonumber\\[0.2cm]
&&  +\varepsilon(X, Y_n)(\omega_n^{1}\oplus \omega_n^{2})(\beta^{2}(y_1), \dots, \beta^{2}(y_{n-1}), (\omega_n^{1}\oplus \omega_n^{2})(\beta(x_1), \dots, \beta(x_{n-1}), (\alpha_{M}^{1}\oplus\alpha_{M}^{2})(m))\nonumber\\[0.2cm]
&=&\displaystyle{\sum_{i=1}^{n-1}}\varepsilon(X, Y_i)\omega_n(\beta^{2}(y_1), \dots, \beta^{2}(y_{i-1}),
[\beta(x_1), \dots, \beta(x_{n-1}), \alpha(y_i)], \beta^{2}(y_{i+1}), \dots, \beta^{2}_M(m))\nonumber\\[0.2cm]
&&+\varepsilon(X, Y_n)\omega_n(\beta^{2}(y_1), \dots, \beta^{2}(y_{n-1}), \omega_n(\beta(x_1), \dots, \beta(x_{n-1}), \alpha_{M}(m))).\nonumber
\end{array}$$
where $x_i, y_j\in\mathcal{H}(L)$, $X=\sum_{i=1}^{n-1}x_i$,  $Y_i=\sum_{j=1}^iy_{j-1}, y_0=e$ and $m\in\mathcal{H}(M)$.
This finishes the proof.
 \end{proof}
\section{Generalized derivation of color Hom-algebras and their color Hom-subalgebras}
\begin{definition}
For any $k\geq 0$, we call $D\in End(L)$ an $(\alpha^k, \beta^r)$-derivation \index{derivation!n-Hom-Lie color algebra} of degree $d$ of the multiplicative $n$-BiHom-Lie color algebra
$(L, [\cdot, \dots, \cdot], \varepsilon, \alpha, \beta)$ if
\begin{eqnarray}
&& \alpha\circ D=D\circ\alpha,\quad \beta\circ D=D\circ\beta, \\
&& D([x_1, \dots, x_n])= \\
&& \sum_{i=1}^n\varepsilon(d, X_{i})[\beta^r\alpha^k(x_1), \dots, \beta^r\alpha^k(x_{i-1}), D(x_i), \beta^r\alpha^k(x_{i+1}), \dots, \beta^r\alpha^k(x_n)].
\label{BakayokoSilvestrov:der} \nonumber
\end{eqnarray}
\end{definition}

We denote the set of $(\alpha^k, \beta^r)$-derivations of the multiplicative $n$-BiHom-Lie color algebra  $L$ by $Der_{(\alpha^k, \beta^r)}(L)$.
For any $D\in Der_{(\alpha^k, \beta^r)}(L)$ and $D'\in Der_{(\alpha^l, \beta^s)}(L)$, let us define their $\varepsilon$-commutator $[D, D']$ as usual:
$$[D, D']=D\circ D'-\varepsilon(d, d')D'\circ D.$$
\begin{lemma}
 For any $D\in Der_{(\alpha^k, \beta^r)}(L)$ and $D'\in Der_{(\alpha^l, \beta^s)}(L)$,
$$[D, D']\in Der_{(\alpha^{k+l}, \beta^{r+s})}(L).$$
\end{lemma}
\begin{proof}
For any $D\in Der_{(\alpha^k, \beta^r)}(L)$ and $D'\in Der_{(\alpha^l, \beta^s)}(L)$ and any $x_1, \dots, x_{n-1}, x_n\in\mathcal{H}(L)$,
\begin{eqnarray}
&&\qquad\qquad\qquad [D, D']([x_1, \dots, x_n])=\nonumber\\
&&=(DD'-\varepsilon(d, d')D'D)([x_1, \dots, x_n])\nonumber\\
&&=DD'([x_1, \dots, x_n])-\varepsilon(d, d')D'D([x_1, \dots, x_n])\nonumber\\
&&=\sum_{i=1}^n\varepsilon(d', X_i)D[\beta^s\alpha^s(x_1), \dots, D'(x_i), \dots, \beta^s\alpha^s(x_n)]
-\sum_{i=1}^n\varepsilon(d, X_i)D'[\beta^r\alpha^k(x_1), \dots, D(x_i), \dots, \beta^r\alpha^k(x_n)]\nonumber\\
&&=\sum_{i=2}^n\sum_{j<i, j=1}^n\varepsilon(d', X_i)\varepsilon(d, X_j)[\beta^{r+s}\alpha^{k+l}(x_1), \dots, D\beta^{l}\alpha^s(x_j), \dots, \beta^{r}\alpha^kD'(x_i), \dots, \beta^{r+s}\alpha^{k+l}(x_n)]\nonumber\\
&&\qquad\qquad+\sum_{i=1}^n\varepsilon(d', X_i)\varepsilon(d, X_i)[\beta^{r+s}\alpha^{k+l}(x_1), \dots, DD'(x_i), \dots,
\beta^{r+s}\alpha^{k+l}(x_n)]\nonumber\\
&&\qquad+\sum_{i=1}^n\sum_{j>i, j=2}^n\varepsilon(d', X_i)\varepsilon(d, d'+X_j)[\beta^{r+s}\alpha^{k+l}(x_1), \dots, \beta^{r}\alpha^kD'(x_i),
\dots, D\beta^{l}\alpha^s(x_j), \beta^{r+s}\alpha^{k+l}(x_n)]\nonumber\\
&&\qquad\qquad-\sum_{i=2}^n\sum_{j<i, j=1}^n\varepsilon(d, d')\varepsilon(d, X_i)\varepsilon(d', X_j)[\beta^{r+s}\alpha^{k+l}(x_1), \dots,
 D'\beta^{l}\alpha^s(x_j), \dots, \beta^{r}\alpha^kD(x_i), \nonumber\\
&&\qquad\dots, \beta^{r+s}\alpha^{k+l}(x_n)]\nonumber\\
&&\qquad-\sum_{i=1}^n\varepsilon(d, d')\varepsilon(d, X_i)\varepsilon(d', X_i)[\beta^{r+s}\alpha^{k+l}(x_1), \dots, D'D(x_i), \dots,
\beta^{r+s}\alpha^{k+l}(x_n)]\nonumber\\
&&\qquad\qquad-\sum_{i=1}^n\sum_{j>i, j=2}^n\varepsilon(d, d')\varepsilon(d, X_i)\varepsilon(d', d+X_j)[\beta^{r+s}\alpha^{k+l}(x_1),
\dots, \beta^{l}\alpha^sD(x_i), \dots, D'\beta^{r}\alpha^k(x_j),\nonumber\\
&&\qquad \dots, \beta^{r+s}\alpha^{k+l}(x_n)]\nonumber.\nonumber
\end{eqnarray}
However,
\begin{eqnarray}
&&\qquad\sum_{i=1}^n\sum_{j>i, j=2}^n\varepsilon(d, d')\varepsilon(d, X_i)\varepsilon(d', d+X_j)
[\beta^{r+s}\alpha^{k+l}(x_1), \dots, \beta^{l}\alpha^sD(x_i), \dots, D'\beta^{r}\alpha^k(x_j), \dots, \beta^{r+s}\alpha^{k+l}(x_n)]=\nonumber\\
&&=\sum_{i=1}^n\sum_{j>i, j=2}^n\varepsilon(d, X_i)\varepsilon(d', X_j)
[\beta^{r+s}\alpha^{k+l}(x_1), \dots, \beta^{l}\alpha^sD(x_i), \dots, D'\beta^{r}\alpha^k(x_j), \dots, \beta^{r+s}\alpha^{k+l}(x_n)]\nonumber\\
&&=\sum_{j=1}^n\sum_{j<i, i=2}^n\varepsilon(d, X_j)\varepsilon(d', X_i)
[\beta^{r+s}\alpha^{k+l}(x_1), \dots, \beta^{l}\alpha^sD(x_j), \dots, D'\beta^{r}\alpha^k(x_i), \dots, \beta^{r+s}\alpha^{k+l}(x_n)]\nonumber\\
&&=\sum_{i=2}^n\sum_{j<i, j=1}^n\varepsilon(d, X_j)\varepsilon(d', X_i)
[\beta^{r+s}\alpha^{k+l}(x_1), \dots, \beta^{l}\alpha^sD(x_j), \dots, D'\beta^{r}\alpha^k(x_i), \dots, \beta^{r+s}\alpha^{k+l}(x_n)]\nonumber.
\end{eqnarray}
Thus,
\begin{eqnarray}
 [D, D']([x_1, \dots, x_n])
&=&\sum_{i=1}^n\varepsilon(d+d', X_i)\Big([\beta^{r+s}\alpha^{k+l}(x_1), \dots, DD'(x_i), \dots, \beta^{r+s}\alpha^{k+l}(x_n)]\nonumber\\
&&-\varepsilon(d, d')[\beta^{r+s}\alpha^{k+l}(x_1), \dots, D'D(x_i), \dots, \beta^{r+s}\alpha^{k+l}(x_n)]\Big)\nonumber\\
&=&\sum_{i=1}^n\varepsilon(d+d', X_i)[\beta^{r+s}\alpha^{k+l}(x_1), \dots, [D, D'](x_i), \dots, \beta^{r+s}\alpha^{k+l}(x_n)]\nonumber,
\end{eqnarray}
which means that $[D, D']\in Der_{(\alpha^{k+l}, \beta^{r+s})}(L)$.
\end{proof}

\begin{lemma}\cite{} \label{asslie}
 If $A, \cdot, \varepsilon, \alpha, \beta)$ is a regular BiHom-associative color algbebra, then
$$L(A)=(A, [-, -], \varepsilon, \alpha, \beta)$$
is also a regular BiHom-Lie color algebra with
$$[x, y]=x\cdot y-\varepsilon(x, y)(\alpha^{-1}\beta(y))(\alpha\beta^{-1}(x)),$$
for any $x, y\in\mathcal{H}(A)$.
\end{lemma}

Let us denote by  $Der(L)=\oplus_{k, r\geq 0}Der_{(\alpha^k, \beta^{r})}(L)$ and $ \widehat{Der}(L)=\{D\in Der(L), D\circ\alpha=\beta\circ D\}$.
\begin{proposition}
$(L, [\cdot, \dots, \cdot], \varepsilon, \alpha, \beta)$ be a regular $n$-BiHom-Lie color algebra. Then,
$(\widehat{Der}(L), [\cdot, \cdot], \varepsilon, \omega, \Omega)$ is a BiHom-Lie color algebra, with $\omega (D)=D\circ \alpha$ and $\Omega (D)=D\circ \beta$.
\end{proposition}
\begin{proof}
 It follows from the simple fact that $( \widehat{Der}(L), \circ, \varepsilon, \omega, \Omega)$ is a BiHom-associative color algebra and
Lemma \ref{asslie}.
\end{proof}

\begin{proposition}
 $(L, [\cdot, \dots, \cdot], \varepsilon, \alpha, \beta)$ be an $n$-BiHom-Lie color algebra and $\widetilde{End}(L)$ the subspace of $End(L)$ defined
by $$\widetilde{End}(L)=\{f\in End(L)/ f\circ\alpha=\alpha\circ f, f\circ \beta=\beta\circ f\}.$$
Then, $(\widetilde{End}(L), [-, -], \varepsilon, \omega, \Omega)$ is a Bihom-lie color algeba with $\omega(D)=D\circ\alpha, \Omega(D)=D\circ\beta$
and $[D, D']=DD'-\varepsilon(d, d')D'D$.
In particular, $(Der(L), [-, -], \varepsilon, \omega, \Omega)$ is a Bihom-Lie color algebra.
\end{proposition}

\begin{definition}
 An endomorphism $D$ of degree $d$ of a multiplicative $n$-BiHom-Lie color algebra $(L, [\cdot, \dots, \cdot], \varepsilon, \alpha, \beta)$ is called a generalized
$(\alpha^k, \beta^r)$-derivation \index{derivation!generalised!multiplicative n-Hom-Lie color algebra} if there exist linear mappings $D', D'',
 \dots,\\ D^{(n-1)}, D^{(n)}$ of degree $d$ such that for any $x_1,\dots,x_n\in\mathcal{H}(L)$:
\begin{eqnarray}
&& D\circ\alpha=\alpha\circ D,\;\;D\circ\beta=\beta\circ D \;\; D^{(i)}\circ\alpha=\alpha\circ D^{(i)}\;\;\mbox{and}\;\; D^{(i)}\circ\beta=\beta\circ D^{(i)},\\
&& \begin{array}{l}
D^{(n)}([x_1, \dots, x_n])= \\
\displaystyle \sum_{i=1}^n\varepsilon(d, X_{i})[\alpha^k\beta^r(x_1), \dots, \alpha^k\beta^r(x_{i-1}), D^{(i-1)}(x_i), \alpha^k\beta^r(x_{i+1}),
 \dots, \alpha^k\beta^r(x_n)]. \label{alphaderivation}
\end{array}
\end{eqnarray}
An $(n+1)$-tuple $(D, D', D'', \dots, D^{(n-1)}, D^{(n)})$ is called an $(n+1)$-ary $(\alpha^k, \beta^r)$-derivation.
\end{definition}
The set of generalized $(\alpha^k, \beta^r)$-derivation is denoted by $GDer_{(\alpha^k, \beta^r)}(L)$.
Set $$GDer(L)=\oplus_{k\geq 0}\oplus_{r\geq 0}GDer_{(\alpha^k, \beta^r)}(L).$$

\begin{definition}
Let$(L, [\cdot, \dots, \cdot], \varepsilon, \alpha, \beta)$ be a multiplicative $n$-BiHom-Lie color algebra . A linear mapping $D\in End(L)$ is said to be an
$(\alpha^k, \beta^r)$-quasiderivation of degree $d$  \index{quasiderivation!multiplicative n-Hom-Lie color algebra} if there exists a $D'\in End(L)$ of degree $d$ such that
\begin{eqnarray}
 &&D\circ\alpha=\alpha\circ D,\quad D'\circ\alpha=\alpha\circ D',\\
&&D\circ\beta=\beta\circ D,\quad D'\circ\beta=\beta\circ D',\\
% \end{eqnarray}
% \begin{eqnarray}
&& D^{'}([x_1, \dots, x_n]) =
 \sum_{i=1}^n\varepsilon(d, X_{i})[\alpha^k\beta^r(x_1), \dots, \alpha^k\beta^r(x_{i-1}), D(x_i), \alpha^k\beta^r(x_{i+1}), \dots, \alpha^k\beta^r(x_n)]
\end{eqnarray}
for all $x_1,\dots,x_n\in\mathcal{H}(L)$
\end{definition}
We call $D'$ the endomorphism associated to the $(\alpha^k, \beta)$-quasiderivation $D$.
The set of ${(\alpha^k, \beta^r)}$-quasiderivations will be denoted $QDer_{(\alpha^k, \beta^r)}(L)$. Set $QDer(L)=\oplus_{k\geq 0}\oplus_{r\geq 0}QDer_{(\alpha^k, \beta^r)}(L)$.

\begin{definition}\label{cet}
 Let$(L, [\cdot, \dots, \cdot], \varepsilon, \alpha, \beta)$ be a multiplicative $n$-BiHom-Lie color algebra . The set $C_{(\alpha^k, \beta^r)}(L)$ consisting of linear mapping
 $D$ of degree $d$ with
the property
\begin{eqnarray}
&&D\circ\alpha=\alpha\circ D,\quad D\circ\beta=\beta\circ D,\\
&& D([x_1, \dots, x_n])=
 \varepsilon(d, X_{i})[\alpha^k\beta^r(x_1), \dots, \alpha^k\beta^r(x_{i-1}), D(x_i), \alpha^k\beta^r(x_{i+1}), \dots, \alpha^k\beta^r(x_n)]
\label{alphacentroid}
\end{eqnarray}
for all $x_1,\dots,x_n\in\mathcal{H}(L)$, is called the $(\alpha^k, \beta^r)$-centroid of $L$.
\index{centroid!multiplicative n-Hom-Lie color algebra}
\end{definition}

\begin{definition}
 Let$(L, [\cdot, \dots, \cdot], \varepsilon, \alpha, \beta)$ be a multiplicative $n$-BiHom-Lie color algebra . The set $QC_{(\alpha^k, \beta^r)}(L)$ consisting of linear mapping
 $D$ of degree $d$ with
the property
\begin{eqnarray}
&&D\circ\alpha=\alpha\circ D, \quad D\circ\beta=\beta\circ D\\
&& [D(x_1), \alpha^k\beta^r(x_2), \dots, \alpha^k\beta^r(x_n)]= \\
&& \varepsilon(d, X_{i})[\alpha^k\beta^r(x_1), \dots, \alpha^k\beta^r(x_{i-1}), D(x_i), \alpha^k\beta^r(x_{i+1}), \dots, \alpha^k\beta^r(x_n)],
\nonumber
\end{eqnarray}
for all $x_1,\dots,x_n\in\mathcal{H}(L)$, is called the $(\alpha^k, \beta^r)$-quasicentroid of $L$.
\index{quasicentroid!multiplicative n-Hom-Lie color algebra}
\end{definition}

\begin{definition}
  Let$(L, [\cdot, \dots, \cdot], \varepsilon, \alpha, \beta)$ be a multiplicative $n$-BiHom-Lie color algebra. The set $ZDer_{(\alpha^k, \beta^r)}(L)$ consisting of linear mappings $D$ of degree $d$, such that for all $x_1,\dots,x_n\in\mathcal{H}(L)$:
\begin{eqnarray}
&&D\circ\alpha=\alpha\circ D, \quad D\circ\beta=\beta\circ D\nonumber\\
&& D([x_1, \dots, x_n])=  \varepsilon(d, X_i)[\alpha^k\beta^r(x_1), \dots, \alpha^k\beta^r(x_{i-1}), D(x_i), \alpha^k\beta^r(x_{i+1}), \dots, \alpha^k\beta^r(x_n)]=0, \nonumber \\
&& \hspace{6cm} i=1, 2, \dots, n, \label{alphacentralderivation},
\end{eqnarray}
is called the set of central $(\alpha^k, \beta^r)$-derivations of $L$. \index{derivation!central!multiplicative n-Hom-Lie color algebra}
\end{definition}

It is easy to see that
$$ZDer(L)\subseteq Der(L)\subseteq QDer(L)\subseteq GDer(L)\subseteq \widetilde{End}(L).$$

\begin{proposition}
 Let$(L, [\cdot, \dots, \cdot], \varepsilon, \alpha, \beta)$ be a regular $n$-BiHom-Lie color algebra .
\begin{enumerate}
 \item [1)] ${GDer}(L), {QDer}(L)$, ${C}(L)$ are color BiHom-subalgebras of
 $(\widetilde{End}(L), [\cdot, \cdot], \omega, \Omega)$ i.e. :
\begin{enumerate}
\item[1a)] $\omega({GDer}(L))\subseteq {GDer}(L)$, $\Omega({GDer}(L))\subseteq {GDer}(L)$
and $[{GDer}(L), {GDer}(L)]\subseteq {GDer}(L)$.
\item[2b)] $\omega({QDer}(L))\subseteq {QDer}(L)$, $\Omega({QDer}(L))\subseteq {QDer}(L)$ and
$[{QDer}(L), {QDer}(L)]\subseteq {QDer}(L)$.
\item[3c)] $\omega({C}(L))\subseteq {C}(L)$, $\Omega({C}(L))\subseteq {C}(L)$ and
 $[{C}(L), {C}(L)]\subseteq {C}(L)$.
\end{enumerate}
\item [2)] ${ZDer}(L)$ is a color BiHom-ideal of ${Der}(L)$ i.e. : \\
$\omega({ZDer}(L))\subseteq {ZDer}(L)$, $\Omega({ZDer}(L))\subseteq {ZDer}(L)$ and
$[{ZDer}(L), {Der}(L)]\subseteq {ZDer}(L)$.
\end{enumerate}
\end{proposition}
\begin{proof}
1a) Let us prove that if $D\in {GDer}(L)$, then $\omega(D)\in {GDer}(L)$.
 For any $x_1,\dots,x_n\in\mathcal{H}(L)$,
\begin{eqnarray}
 &&(\omega(D^{(n)}))([x_1, \dots, x_n])=(D^{(n)}\circ\alpha)([x_1, \dots, x_i, \dots, x_n]) \nonumber \\
&&=D^{(n)}([\alpha(x_1), \dots, \alpha(x_i), \dots, \alpha(x_n)])\nonumber\\
&&=\sum_{i=1}^n\varepsilon(d, X_{i})[\alpha^{k+1}\beta^r(x_1), \dots, \alpha^{k+1}\beta^r(x_{i-1}), D^{(i-1)}\alpha(x_i),
\alpha^{k+1}\beta^r(x_{i+1}), \dots,
\alpha^{k+1}\beta^r(x_n)]\nonumber\\
&&=\sum_{i=1}^n\varepsilon(d, X_{i})[\alpha^{k+1}\beta^r(x_1), \dots, \alpha^{k+1}\beta^r(x_{i-1}), ( D^{(i-1)}\circ\alpha)(x_i), \nonumber \\
&& \hspace{6cm} \alpha^{k+1}\beta^r(x_{i+1}), \dots,\alpha^{k+1}\beta^r(x_n)]\nonumber\\
&&=\sum_{i=1}^n\varepsilon(d, X_{i})[\alpha^{k+1}\beta^r(x_1), \dots, \alpha^{k+1}\beta^r(x_{i-1}), \omega(D^{(i-1)})(x_i),
\alpha^{k+1}\beta^r(x_{i+1}), \dots,
\alpha^{k+1}\beta^r(x_n)]\nonumber.
\end{eqnarray}
This means that $\omega(D)$ is an $(\alpha^{k+1}, \beta^r)$-derivation i.e. $\omega(D)\in {GDer}(L)$.\\
Similarly, we prove that $\Omega(D)$ is an $(\alpha^{r}, \beta^{r+1})$-derivation.\\
Now let $D_1\in {GDer}_{(\alpha^{k}, \beta^r)}(L)$ and $D_2\in {GDer}_{(\alpha^{l}, \beta^s)}(L)$, we have
\begin{eqnarray}
 &&(D_2^{(n)}D_1^{(n)})([x_1, \dots, x_n])=D_2^{(n)}(D_1^{(n)}([x_1, \dots, x_n]))=\nonumber\\
&&=\sum_{i=1}^n\varepsilon(d_1, X_{i})D_2^{(n)}([\alpha^{k}\beta^r(x_1), \dots, \alpha^{k}\beta^r(x_{i-1}), D_1^{(i-1)}(x_i),
 \alpha^{k}\beta^r(x_{i+1}), \dots,
\alpha^{k}\beta^r(x_n)])\nonumber\\
&&=\sum_{i=2}^n\sum_{j<i}^n\varepsilon(d_1, X_{i})\varepsilon(d_2, X_{j})
[\alpha^{k+l}\beta^{r+s}(x_1), \dots, D_2^{(j-1)}\alpha^{k}\beta^{r}(x_j),\dots, \alpha^{k+l}\beta^{r+s}(x_{i-1}),
D_1^{(i-1)}\alpha^{l}\beta^{s}(x_i),
\nonumber \\
&& \hspace{7cm}
\alpha^{k+l}\beta^{r+s}(x_{i+1}), \dots,
\alpha^{k+l}\beta^{r+s}(x_n)]\nonumber\\
&&+\sum_{i=1}^n\varepsilon(d_1+d_2, X_{i})
[\alpha^{k+l}\beta^{r+s}(x_1), \dots,  \alpha^{k+l}\beta^{r+s}(x_{i-1}), D_2^{(i-1)}D_1^{(i-1)}(x_i), \nonumber \\
&& \hspace{7cm}
\alpha^{k+l}\beta^{r+s}(x_{i+1}), \dots,
\alpha^{k+l}\beta^{r+s}(x_n)]\nonumber\\
&&+\sum_{i=1}^n\sum_{j>i}^n\varepsilon(d_1, X_{i})\varepsilon(d_2, d_1+X_{j})
[\alpha^{k+l}\beta^{r+s}(x_1), ,\dots, \alpha^{k+l}\beta^{r+s}(x_{i-1}), D_1^{(i-1)}\alpha^{l}\beta^s(x_i), \nonumber \\
&& \hspace{5cm}
\alpha^{k+l}\beta^{r+s}(x_{i+1}), \dots, D_2^{(j-1)}\alpha^{k}\beta^{r}(x_j), \dots,
\alpha^{k+l}\beta^{r+s}(x_n)]\nonumber.
\end{eqnarray}
It follows that
\begin{eqnarray}
 &&([D_1^{(n)}, D_2^{(n)}])([x_1, \dots, x_n])=(D_1^{(n)}D_2^{(n)}-\varepsilon(d_1, d_2)D_2^{(n)}D_1^{(n)})([x_1, \dots, x_n])=\nonumber\\
&&=\sum_{i=1}^n\varepsilon(d_1+d_2, X_{i})
[\alpha^{k+l}\beta^{r+s}(x_1), \dots,  \alpha^{k+l}\beta^{r+s}(x_{i-1}),
\nonumber \\
&& \hspace{1.5cm} (D_1^{(i-1)}D_2^{(i-1)}-\varepsilon(d_1, d_2)D_2^{(i-1)}D_1^{(i-1)})(x_i), \alpha^{k+l}\beta^{r+s}(x_{i+1}), \dots,
\alpha^{k+l}\beta^{r+s}(x_n)]\nonumber\\
&&=\sum_{i=1}^n\varepsilon(d_1+d_2, X_{i})
[\alpha^{k+l}\beta^{r+s}(x_1), \dots,  \alpha^{k+l}\beta^{r+s}(x_{i-1}),
\nonumber \\
&& \hspace{1.5cm}
[D_1^{(i-1)}, D_2^{(i-1)}](x_i), \alpha^{k+l}\beta^{r+s}(x_{i+1}), \dots,
\alpha^{k+l}\beta^{r+s}(x_n)]\nonumber.
\end{eqnarray}
Thus we obtain that $[D_1, D_2]\in {GDer}_{(\alpha^{k+l}, \beta^{r+s})}(L)$.
\begin{enumerate}
\item[1b)] That of ${QDer}(L)$ is a color BiHom-subalgebra of $(\widetilde{End} (L), [\cdot, \cdot], \omega, \Omega)$ is done in the similar way.
\item[1c)]  Let $D_1\in {C}_{(\alpha^{k}, \beta^r)}(L)$ and $D_2\in {C}_{(\alpha^{l}, \beta^s)}(L)$. Then
\begin{eqnarray}
&& \omega(D_1)([x_1, x_2, \dots, x_n])=\alpha D_1([x_1, x_2, \dots, x_n])\nonumber\\
&&=\varepsilon(d_1, X_i)\alpha([\alpha^k\beta^r(x_1), \alpha^k\beta^r(x_2), \dots, D_1(x_i), \dots, \alpha^k\beta^r(x_n])\nonumber\\
&&=\varepsilon(d_1, X_i)[\alpha^{k+1}\beta^r(x_1), \alpha^{k+1}\beta^r(x_2), \dots, \alpha D_1(x_i), \dots, \alpha^{k+1}\beta^r(x_n]\nonumber.
\end{eqnarray}
Thus $\omega(D)\in {C}_{(\alpha^{k+1}, \beta^r)}(L)$. In the same way, $\Omega(D)\in {C}_{(\alpha^{k}, \beta^{r+1})}(L)$.\\
 Moreover,
\begin{eqnarray}
&& [D_1, D_2]([x_1, \dots, x_n])=D_1D_2([x_1, \dots, x_n])-
\varepsilon(d_1, d_2)D_2D_1([x_1, \dots, x_n])\nonumber\\
&&=\varepsilon(d_2, X_i)D_1[\beta^s\alpha^l(x_1), \beta^s\alpha^l(x_2), \dots, D_2(x_i), \dots, \beta^s\alpha^l(x_n)]\nonumber\\
&&\quad-\varepsilon(d_1, d_2)\varepsilon(d_1, X_i)D_2[\beta^s\alpha^l(x_1), \beta^s\alpha^l(x_2), \dots, D_1(x_i), \dots, \beta^s\alpha^l(x_n)]\nonumber\\
&&=\varepsilon(d_1+d_2, X_i)[\alpha^{k+l}\beta^{r+s}(x_1), \alpha^{k+l}\beta^{r+s}(x_2), \dots,D_1 D_2(x_i), \dots, \alpha^{k+l}\beta^{r+s}(x_n)]\nonumber\\
&&\quad-\varepsilon(d_1+d_2, X_i)[\alpha^{k+l}\beta^{r+s}(x_1), \alpha^{k+l}\beta^{r+s}(x_2), \varepsilon(d_1, d_2)\dots,D_2 D_1(x_i), \dots, \alpha^{k+l}\beta^{r+s}(x_n)]\nonumber\\
&&=\varepsilon(d_1+d_2, X_i)[\alpha^{k+l}\beta^{r+s}(x_1), \alpha^{k+l}\beta^{r+s}(x_2), \dots, [D_1 D_2](x_i), \dots, \alpha^{k+l}\beta^{r+s}(x_n)]\nonumber.
\end{eqnarray}
So, $[D_1, D_2]\in {C}_{(\alpha^{k+l}, \beta^{r+s})}(L)$ and finally $[D_1, D_2]\in {C}(L)$.
\item[2)] By the same method as previously one can show that $\omega(D)\in {ZDer}_{(\alpha^{k+1}, \beta^r)}(L)$,\\
 $\Omega(D)\in {ZDer}_{(\alpha^k, \beta^{r+1})}(L)$ and
 $[D_1, D_2]\in {ZDer}_{(\alpha^{k+l}, \beta^{r+s})}(L)$, where $D_1\in {ZDer}_{(\alpha^{k}, \beta^{r})}(L)$
and $D_2\in {Der}_{(\alpha^{l}, \beta^{s})}(L)$.
\end{enumerate}
\end{proof}

\begin{lemma}\label{lm}
 Let$(L, [\cdot, \dots, \cdot], \varepsilon, \alpha, \beta)$ be a multiplicative $n$-BiHom-Lie color algebra . Then
\begin{enumerate}
\item [1)] The $\varepsilon$-commutator of two elements of quasicentroid is a quasiderivation i.e.
 $$[QC(L), QC(L)]\subseteq QDer(L).$$
\item [2)] $QDer(L)+QC(L)\subseteq GDer(L)$.
\end{enumerate}
\end{lemma}
\begin{proof}
 For any $x_1, x_2, \dots, x_n\in\mathcal{H}(L)$,
\begin{enumerate}
 \item [1)] let $D_1\in QC_{(\alpha^{k}, \beta^r)}(L)$ and $D_2\in QC_{(\alpha^{l}, \beta^s)}(L)$. We have, on the one hand
\begin{eqnarray}
&&[D_1D_2(x_1), \alpha^{k+l}\beta^{r+s}(x_2), \dots, \alpha^{k+l}\beta^{r+s}(x_n)]\nonumber\\
&&=\varepsilon(D_1, D_2+X_i)[D_2(\beta^{r}\alpha^{k}(x_1)), \alpha^{k+l}\beta^{r+s}(x_2), \dots, D_1(\beta^l\alpha^{s}(x_i)), \dots, \alpha^{k+l}\beta^{r+s}(x_n)]\nonumber\\
&&=\varepsilon(D_1, D_2+X_i)\varepsilon(D_2, X_i)[\alpha^{k+l}\beta^{r+s}(x_1), \dots, D_2D_1(x_i), \dots, \alpha^{k+l}\beta^{r+s}(x_n)]\nonumber\\
&&=\varepsilon(D_1, D_2)\varepsilon(D_1+D_2, X_i)[\alpha^{k+l}\beta^{r+s}(x_1), \dots, D_2D_1(x_i), \dots, \alpha^{k+l}\beta^{r+s}(x_n)]\nonumber.
\end{eqnarray}
On the other hand,
\begin{eqnarray}
 &&[D_1D_2(x_1), \alpha^{k+l}\beta^{r+s}(x_2), \dots, \alpha^{k+l}\beta^{r+s}(x_n)]=\nonumber\\
&&=\varepsilon(D_1, D_2+x_1)[D_2(\beta^r\alpha^{k}(x_1)), D_1(\beta^l\alpha^{s}(x_2)), \dots, \alpha^{k+l}\beta^{r+s}(x_i), \dots, \alpha^{k+l}\beta^{r+s}(x_n)]\nonumber\\
&&=\varepsilon(D_1, D_2+x_1)\varepsilon(D_2, D_1+X_i)
\nonumber \\
&& \hspace{3cm}
[\alpha^{k+l}\beta^{r+s}(x_1), D_1(\beta^l\alpha^{s}(x_2)), \dots, D_2(\beta^r\alpha^{k}(x_i)), \dots, \alpha^{k+l}\beta^{r+s}(x_n)]\nonumber\\
&&=\varepsilon(D_1, x_1)\varepsilon(D_2, X_i)\varepsilon(x_1, D_1)
\nonumber \\
&&
[D_1(\beta^l\alpha^{s}(x_1)), \alpha^{k+l}\beta^{r+s}(x_2), \dots, D_2(\beta^r\alpha^{k}(x_i)), \dots, \alpha^{k+l}\beta^{r+s}(x_n)]\nonumber\\
&&=\varepsilon(D_2, X_i)\varepsilon(D_1, X_i)[\alpha^{k+l}\beta^{r+s}(x_1), \dots, D_1D_2(x_i), \dots, \alpha^{k+l}\beta^{r+s}(x_n)]\nonumber\\
&&=\varepsilon(D_1+D_2, X_i)[\alpha^{k+l}\beta^{r+s}(x_1), \dots, D_1D_2(x_i), \dots, \alpha^{k+l}\beta^{r+s}(x_n)],\nonumber
\end{eqnarray}
and so
\begin{eqnarray}
&&\varepsilon(D_1+D_2, X_i)[\alpha^{k+l}\beta^{r+s}(x_1), \dots, [D_1, D_2](x_i), \dots, \alpha^{k+l}\beta^{r+s}(x_n)]=\nonumber\\
&&=\varepsilon(D_1+D_2, X_i)\Big([\alpha^{k+l}\beta^{r+s}(x_1), \dots, D_1D_2(x_i), \dots, \alpha^{k+l}\beta^{r+s}(x_n)]\nonumber\\
&&\quad-\varepsilon(D_1, D_2)[\alpha^{k+l}\beta^{r+s}(x_1), \dots, D_2D_1(x_i), \dots, \alpha^{k+l}\beta^{r+s}(x_n)]\Big)=0.\nonumber
\end{eqnarray}
It follows that
 \begin{eqnarray}
  \sum_{i=1}^n\varepsilon(D_1+D_2, X_i)[\alpha^{k+l}\beta^{r+s}(x_1), \dots, [D_1, D_2](x_i), \dots, \alpha^{k+l}\beta^{r+s}(x_n)]=0\nonumber.
 \end{eqnarray}
Therefore $D'\equiv 0$, and $[D_1, D_2]\in QDer(L)$.
\item [2)] let $D_1\in QDer_{(\alpha^k, \beta^r)}(L)$ and $D_2\in QC_{(\alpha^l, \beta^s)}(L)$ with $|D_1|=|D_2|$. Then there exists $D'_1\in End(L)$ such that
\begin{eqnarray}
&& D'_1([x_1, \dots, x_n])
=\sum_{i=1}^n \varepsilon(D_1, X_i)[\beta^r\alpha^{k}(x_1), \dots, D_1(x_i), \dots, \beta^r\alpha^{k}(x_n)]\nonumber\\
&&=[D_1(x_1), \beta^r\alpha^{k}(x_2), \dots, \beta^r\alpha^{k}(x_n)]
+\varepsilon(D_1, x_1)[\beta^r\alpha^{k}(x_1), D_1(x_2), \dots, \beta^r\alpha^{k}(x_n)]\nonumber\\
&& \qquad +\sum_{i=3}^n\varepsilon(D_1, X_i)[\beta^r\alpha^{k}(x_1), \dots, D_1(x_i), \dots, \beta^r\alpha^{k}(x_n)]\nonumber\\
&&=[(D_1+D_2)(x_1), \beta^r\alpha^{k}(x_2), \dots, \beta^r\alpha^{k}(x_n)]-[D_2(x_1), \beta^r\alpha^{k}(x_2), \dots, \beta^r\alpha^{k}(x_n)]\nonumber\\
&& \qquad +\varepsilon(D_1, x_1)[\beta^r\alpha^{k}(x_1), D_1(x_2), \dots, \beta^r\alpha^{k}(x_n)]
\nonumber \\
&& \qquad + \sum_{i=3}^n\varepsilon(D_1, X_i)[\beta^r\alpha^{k}(x_1), \dots, D_1(x_i), \dots, \beta^r\alpha^{k}(x_n)]\nonumber\\
&&=[(D_1+D_2)(x_1), \beta^r\alpha^{k}(x_2), \dots, \beta^r\alpha^{k}(x_n)]-\varepsilon(D_2, x_1)[\beta^r\alpha^{k}(x_1), D_2(x_2), \dots, \beta^r\alpha^{k}(x_n)]\nonumber\\
&& \quad +\varepsilon(D_1, x_1)[\beta^r\alpha^{k}(x_1), D_1(x_2), \dots, \beta^r\alpha^{k}(x_n)]
\nonumber \\
&& \qquad + \sum_{i=3}^n\varepsilon(D_1, X_i)[\beta^r\alpha^{k}(x_1), \dots, D_1(x_i), \dots, \beta^r\alpha^{k}(x_n)]\nonumber\\
&&=[(D_1+D_2)(x_1), \beta^r\alpha^{k}(x_2), \dots, \beta^r\alpha^{k}(x_n)]
\nonumber\\
&& \qquad
+\varepsilon(D_2, x_1)[\beta^r\alpha^{k}(x_1), (D_1-D_2)(x_2), \dots, \beta^r\alpha^{k}(x_n)]
\nonumber\\
&& \qquad +\sum_{i=3}^n\varepsilon(D_1, X_i)[\beta^r\alpha^{k}(x_1), \dots, D_1(x_i), \dots, \beta^r\alpha^{k}(x_n)]\nonumber.
\end{eqnarray}
\end{enumerate}
The conclusion follows by taking
$$D^{(n)}=D'_1,\quad D=D_1+D_2,\quad D'=D_1-D_2,\quad D^{(i)}=D_1,\quad 2\leq i\leq n-1.$$
This proved that $D_1+D_2\in GDe(L)$.
\end{proof}

% \begin{proposition}
% If$(L, [\cdot, \dots, \cdot], \varepsilon, \alpha, \beta)$ is a multiplicative $n$-BiHom-Lie color algebra , then
% $$QC(L)+[QC(L), QC(L)]$$
%  is a color BiHom-subalgebra of $GDer(L)$.
% \end{proposition}
%  \begin{proof}
%  It follows from Lemma \ref{lm} by using the same arguments as in Proposition 2.4  in \cite{ChapBaSilnhomliecolor:KP}.
%  \end{proof}

\begin{proposition}
 Let $(L, [\cdot, \dots, \cdot], \varepsilon, \alpha, \beta)$ be a multiplicative $n$-BiHom-Lie color algebra  such that $\alpha$ and $\beta$ be
surjective mappings, then
$[C(L), QC(L)]\subseteq Hom(L, Z(L))$. Moreover, if $Z(L)=\{0\}$, then $[C(L), QC(L)]=\{0\}$.
\end{proposition}
\begin{proof}
 Let $D_1\in C_{(\alpha^k, \beta^r)}(L)$,  $D_2\in QC_{(\alpha^l, \beta^s)}(L)$ and $x_1, \dots, x_n\in \mathcal{H}(L)$.
Since $\alpha$ and $\beta$ are surjective, for
any $y'_i\in L$, there exists $y_i\in L$ such that $y'_i=\beta^{r+s}\alpha^{k+l}(y_i), i= 2, \dots, n$. Thus
\begin{eqnarray}
&& [[D_1, D_2](x_1), y'_2, \dots, y'_n]=\nonumber\\
&&=[[D_1, D_2](x_1), \alpha^{k+l}\beta^{r+s}(y_2), \dots, \alpha^{k+l}\beta^{r+s}(y_n)]\nonumber\\
&&=[D_1D_2(x_1), \alpha^{k+l}\beta^{r+s}(y_2), \dots, \alpha^{k+l}\beta^{r+s}(y_n)]
\nonumber\\
&& \qquad
-\varepsilon(d_1, d_2)[D_2D_1(x_1), \alpha^{k+l}\beta^{r+s}(y_2), \dots, \alpha^{k+l}\beta^{r+s}(y_n)]\nonumber\\
&&=D_1([D_2(x_1), \beta^l\alpha^{s}(y_2), \dots, \beta^l\alpha^{s}(y_n)])
\nonumber\\
&& \qquad
-\varepsilon(d_1, d_2)\varepsilon(d_2,x_1+ d_1)[D_1\beta^l\alpha^{s}(x_1), D_2\beta^r\alpha^{k}(y_2), \dots, \alpha^{k+l}\beta^{r+s}(y_n)]\nonumber\\
&&=D_1([D_2(x_1), \beta^l\alpha^{s}(y_2), \dots, \beta^l\alpha^{s}(y_n)])
\nonumber\\
&& \qquad
-\varepsilon(d_2, x_1)D_1[\beta^l\alpha^{s}(x_1), D_2(y_2), \dots, \beta^l\alpha^{s}(y_n)]\nonumber\\
&&=D_1\Big([D_2(x_1), \beta^s\alpha^{l}(y_2), \dots, \beta^s\alpha^{l}(y_n)]
\nonumber\\
&& \qquad
-\varepsilon(d_2, x_1)[\beta^l\alpha^{s}(x_1), \beta^l\alpha^{s}D_2(y_2), \dots, \beta^l\alpha^{s}(y_n)]\Big)\nonumber\\
&&=D_1\Big([D_2(x_1), \beta^l\alpha^{s}(y_2), \dots, \beta^l\alpha^{s}(y_n)]
-[D_2(x_1), \beta^l\alpha^{s}(y_2), \dots, \beta^l\alpha^{s}(y_n)]\Big)=0.\nonumber
\end{eqnarray}
Hence, $[D_1, D_2](x_1)\in Z(L)$, and $[D_1, D_2]\in Hom(L, Z(L))$. Furthermore, if $Z(L)=\{0\}$, we know that $[C(L), QC(L)]=\{0\}$.
 \end{proof}
%  \section{Centroid of $n$-BiHom-Lie color algebra s}

\begin{proposition}
 Let$(L, [\cdot, \dots, \cdot], \varepsilon, \alpha, \beta)$ is a multiplicative $n$-BiHom-Lie color algebra  with surjective twisting $\alpha$
  and $H$ be a graded Hom-subalgebra of $L$ wich is invariant by $\alpha$ and $\beta$. Then
\begin{enumerate}
 \item [i)] $Z_L(H)$ is invariant under $C(L)$.
\item [ii)] Every perfect color Hom-ideal of $L$ is invariant under $C(L)$.
\end{enumerate}
\end{proposition}
\begin{proof}
\begin{enumerate}
 \item [i)]
 For any $\varphi\in C(L)$ and $x\in Z_L(H)$,  we have
\begin{eqnarray}
 0=\varphi([x, H, L, \dots, L])=[\varphi(x), \beta^r\alpha^k(H), \beta^r\alpha^k(L), \dots, \beta^r\alpha^k(L)]=[\varphi(x), H, L, \dots, L].\nonumber
\end{eqnarray}
Therefore $\varphi(x)\in Z_L(H)$, which implies that $Z_L(H)$ is invariant under $C(L)$.
\item [ii)] Let $H$ be a perfect color Hom-ideal of $L$. Then $H^1=H$, and so for any $x\in H$ there exist $x_1^i, x_2^i, \dots, x_n^i\in H$ with
$0<i<\infty$ such that $x=\sum_i[x_1^i, x_2^i, \dots, x_n^i]$. If $\varphi\in C(L)$, then
\begin{eqnarray*}
\varphi(x) &=& \varphi(\sum_i[x_1^i, x_2^i, \dots, x_n^i])=\sum_i\varphi([x_1^i, x_2^i, \dots, x_n^i])
\nonumber \\
&=&
\sum_i[\varphi(x_1^i), \beta^r\alpha^k(x_2^i), \dots, \beta^r\alpha^k(x_n^i)]\in H.
\end{eqnarray*}
This shows that $H$ is invariant under $C(L)$.
\end{enumerate}
\end{proof}

\begin{proposition}\label{BakayokoSilvestrov:dc1}
 If the characteristic of $\mathbb{K}$ is $0$ or not a factor of $n-1$. Then
$$ZDer(L)=C(L)\cap Der(L).$$
\end{proposition}
\begin{proof}
 If $\varphi\in C(L)\cap Der(L)$, then by (\ref{alphaderivation})  we have
$$\varphi([x_1, \dots, x_n])=\sum_{i=1}^n\varepsilon(d, X_i)[\beta^r\alpha^k(x_1), \dots, \varphi(x_i), \dots, \beta^r\alpha^k(x_n)],$$
and by (\ref{alphacentroid}), for $i=1, 2, \dots, n$,
$$\varepsilon(d, X_i)[\beta^r\alpha^k(x_1), \dots, \varphi(x_i), \dots, \beta^r\alpha^k(x_n)]=\varphi([x_1, \dots, x_n]).$$
Thus
$$\varphi([x_1, \dots, x_n])=n \varphi([x_1, \dots, x_n])$$
The characteristic of  $\mathbb{K}$ being $0$ or not a factor of $n-1$, we have
$$0=\varphi([x_1, \dots, x_n])=\varepsilon(d, X_i)[\beta^r\alpha^k(x_1), \dots, \varphi(x_i), \dots, \beta^r\alpha^k(x_n)], i=1, 2, \dots, n.$$
Which means that $\varphi\in ZDer(L)$.\\
Conversly, let $\varphi\in ZDer(L)$, Then by (\ref{alphacentralderivation}),
$$\varphi([x_1, \dots, x_n])=\varepsilon(d, X_i)[\alpha^k(x_1), \dots, \varphi(x_i), \dots, \alpha^k(x_n)]=0, 1\leq i\leq n$$
and thus $\varphi\in C(L)\cap Der(L)$. Therefore $ZDer(L)=C(L)\cap Der(L)$.
 \end{proof}

\begin{proposition}\label{BakayokoSilvestrov:dc2}
 Let $L$ be an $n$-BiHom-Lie color algebra . For any $D\in Der(L)$ and $\varphi\in C(L)$
\begin{enumerate}
 \item[1)] $Der(L)$ is contained in the normalizer of $C(L)$ in $\widetilde{End(L)}$ i.e.
$$[Der(L), C(L)]\subseteq C(L).$$
\item [2)] $QDer(L)$ is contained in the normalizer of $QC(L)$ in $\widetilde{End(L)}$ i.e.
$$[QDer(L), QC(L)]\subseteq QC(L).$$
\end{enumerate}
\end{proposition}
\begin{proof}
1) For any $D\in Der(L), \varphi\in C(L)$ and $x_1, x_2, \dots, x_n\in\mathcal{H}(L)$,
 \begin{eqnarray}
&& D\varphi([x_1, \dots, x_n])=D([\varphi(x_1), \beta^{r}\alpha^{k}(x_2), \dots, \beta^{r}\alpha^{k}(x_i), \dots, \beta^{r}\alpha^{k}(x_n)])\nonumber\\
&&=[D\varphi(x_1), \beta^{r+s}\alpha^{k+l}(x_2), \dots, \beta^{r+s}\alpha^{k+l}(x_i), \dots, \beta^{r+s}\alpha^{k+l}(x_n)]\nonumber\\
&& \quad +\sum_{i=2}^n\varepsilon(d, \varphi+ X_i)[\beta^{l}\alpha^{s}\varphi(x_1), \beta^{r+s}\alpha^{k+l}(x_2), \dots, \beta^{r}\alpha^{k}D(x_i), \dots, \beta^{r+s}\alpha^{k+l}(x_n)]\nonumber\\
&&=[D\varphi(x_1), \beta^{r+s}\alpha^{k+l}(x_2), \dots, \beta^{r+s}\alpha^{k+l}(x_i), \dots, \beta^{r+s}\alpha^{k+l}(x_n)]\nonumber\\
&& \quad +\sum_{i=2}^n\varepsilon(d,\varphi+ X_i)\varepsilon(\varphi, X_i)([\beta^{r+s}\alpha^{k+l}(x_1), \beta^{r+s}\alpha^{k+l}(x_2), \dots, \varphi D(x_i),
\dots, \beta^{r+s}\alpha^{k+l}(x_n)]\nonumber\\
&=& [D\varphi(x_1), \beta^{r+s}\alpha^{k+l}(x_2), \dots, \beta^{r+s}\alpha^{k+l}(x_i), \dots, \beta^{r+s}\alpha^{k+l}(x_n)]\nonumber\\
&& \quad +\varepsilon(d,\varphi)\sum_{i=2}^n\varepsilon(d+\varphi, X_i)([\beta^{r+s}\alpha^{k+l}(x_1), \beta^{r+s}\alpha^{k+l}(x_2), \dots,
\varphi D(x_i), \dots, \beta^{r+s}\alpha^{k+l}(x_n)]\nonumber\\
&=&[D\varphi(x_1), \beta^{r+s}\alpha^{k+l}(x_2), \dots, \beta^{r+s}\alpha^{k+l}(x_i), \dots, \beta^{r+s}\alpha^{k+l}(x_n)]\nonumber\\
&&\quad +\varepsilon(d,\varphi)\Big(\varphi D[x_1, x_2, \dots, x_i, \dots, x_n]\nonumber\\
&&\quad-[\varphi D(x_1), \beta^{r+s}\alpha^{k+l}(x_2), \dots, \beta^{r+s}\alpha^{k+l}(x_i), \dots, \beta^{r+s}\alpha^{k+l}(x_n)]   \Big)\nonumber.
 \end{eqnarray}
Then we get
\begin{eqnarray}
&& (D\varphi-\varepsilon(d, \varphi)\varphi D)([x_1, \dots, x_n])
 \nonumber \\
&&\qquad  =[(D\varphi-\varepsilon(d, \varphi)\varphi D)(x_1), \dots, \beta^{r+s}\alpha^{k+l}(x_2),
 \dots, \beta^{r+s}\alpha^{k+l}(x_i), \dots, \beta^{r+s}\alpha^{k+l}(x_n)]\nonumber,
\end{eqnarray}
that is
$[D, \varphi]=D\varphi-\varepsilon(d, \varphi)\varphi D\in C(L)$. \\
2) It is proved by using a similar method.
 \end{proof}

\begin{proposition}
 Let $L$ be an $n$-BiHom-Lie color algebra . For any $D\in Der(L)$ and $\varphi\in C(L)$
\begin{enumerate}
\item [1)] $D\varphi$ is contained in $C(L)$ if and only if $\varphi D$ is a central derivation of $L$ i.e. :
$$Der(L)\cdot C(L)\subset C(L)\Longleftrightarrow C(L)\cdot Der(L)\subset ZDer(L).$$
\item [2)] $D\varphi$ is a derivation of $L$ if and only if $[D, \varphi]$ is a central derivation of $L$ i.e. :
$$Der(L)\cdot C(L)\subset Der(L)\Longleftrightarrow [Der(L), C(L)]\subset ZDer(L).$$
\end{enumerate}
\end{proposition}
\begin{proof}
1) From Proposition \ref{BakayokoSilvestrov:dc2}, $D\varphi$ is an element of $C(L)$ if and only if $\varphi D\in  Der(L)\cap C(L)$.
Thanks to Proposition \ref{BakayokoSilvestrov:dc1}, we get the result. \\
2) The conclusion follows from $1)$, Proposition \ref{BakayokoSilvestrov:dc1} and Proposition \ref{BakayokoSilvestrov:dc2}.
\end{proof}

\begin{proposition}
 Let$(L, [\cdot, \dots, \cdot], \varepsilon, \alpha, \beta)$ be a multiplicative $n$-BiHom-Lie color algebra.
\begin{enumerate}
\item [1)]
If $\varphi\in C(L)$ and $D\in Der(L)$, then $\varphi D$ is a derivation i.e.
$$C(L)\cdot Der(L)\subseteq Der(L).$$
\item [2)] Any element of centroid is a quasiderivation i.e.
 $$C(L)\subseteq QDer(L).$$
\end{enumerate}
\end{proposition}

\begin{proof}
1) For any $x_1, \dots, x_n\in \mathcal{H}(L)$,
\begin{eqnarray}
 \varphi D([x_1, \dots, x_n])&=&\sum_{i=1}^n\varepsilon(d, X_i)\varphi([\beta^r\alpha^k(x_1), \dots, D(x_i), \dots, \beta^r\alpha^k(x_n)])\nonumber\\
&=&\sum_{i=1}^n\varepsilon(d, X_i)\varepsilon(\varphi, X_i)[\alpha^{k+l}\beta^{r+s}(x_1), \dots, \varphi D(x_i), \dots, \alpha^{k+l}\beta^{r+s}(x_n)]\nonumber\\
&=&\sum_{i=1}^n\varepsilon(d+\varphi, X_i)[\alpha^{k+l}\beta^{r+s}(x_1), \dots, \varphi D(x_i), \dots, \alpha^{k+l}\beta^{r+s}(x_n)]\nonumber.
\end{eqnarray}
Thus $\varphi D$ is an $(\alpha^{k+l}, \beta^{r+s})$-derivation of degree $d+\varphi$.

2) Let $D$ be an $(\alpha^{k}, \beta^r)$-centroid, then for any $x_1, \dots, x_n\in \mathcal{H}(L),$
\begin{eqnarray}
 D([x_1, \dots, x_n])=\varepsilon(d, X_i)[\beta^r\alpha^k(x_1), \dots D(x_i), \dots, \beta^r\alpha^k(x_n)], i=1, 2, \dots, n.
\end{eqnarray}
It follows that
\begin{eqnarray}
 \sum_{i=1}^n\varepsilon(d, X_i)[\beta^r\alpha^k(x_1), \dots D(x_i), \dots, \beta^r\alpha^k(x_n)]=nD([x_1, \dots, x_n]).
\end{eqnarray}
It suffises to take $D'=nD$.
\end{proof}

If $A$ is a commutative associative algbebra and $L$ is an $n$-BiHom-Lie color algebra , the $n$-Hom-Lie algebra $A\otimes L$ (Theorem \ref{BakayokoSilvestrov:tp})
 is called \index{tensor product!n-Hom-Lie color algebra} the tensor product $n$-BiHom-Lie color algebra  of $A$ and $L$. For $f\in End(A)$ and $\varphi\in End(L)$ let
$f\otimes \varphi : A\otimes L\rightarrow A\otimes L$ be given by $f\otimes \varphi(a\otimes x)=f(a)\otimes\varphi(x)$, for $a\in A, x\in L$.
Then $f\otimes\varphi\in End(A\otimes L)$.

Recall that if $A$ is a commutative associative algebra, the centroid $C(A)$ of $A$ is by definition
$$C(A)=\{f\in End(A) \mid  f(ab)=f(a)b=af(b), \forall a, b\in A\}.$$
We now state the following
\begin{proposition}
 By the above notation, we have
$$C(A)\otimes C(L)\subseteq C(A\otimes L).$$
\end{proposition}
\begin{proof}
For any $a_i\in A, x_i\in\mathcal{H}(L), 1\leq i\leq n$, and any $f\in C(A)$ and $\varphi\in C(L),$
\begin{eqnarray}
&& (f\otimes\varphi)[a_1\otimes x_1, \dots, a_n\otimes x_n]=(f\otimes \varphi)(a_1\dots a_n)\otimes[x_1, \dots, x_n]\nonumber\\
&&=f(a_1\dots a_n)\otimes\varphi[x_1, \dots, x_n]\nonumber\\
&&=\varepsilon(\varphi, X_i)a_1\dots f(a_i)\dots a_n\otimes [\beta^r\alpha^{k}(x_1), \dots, \varphi(x_i), \dots, \beta^r\alpha^{k}(x_n)]\nonumber\\
&&=\varepsilon(\varphi, X_i)[a_1\otimes \beta^r\alpha^{k}(x_1)\dots f(a_i)\otimes\varphi(x_i), \dots, a_n\otimes \beta^r\alpha^{k}(x_n)]\nonumber\\
&&=\varepsilon(\varphi, X_i)[\gamma(a_1\otimes x_1)\dots (f\otimes\varphi)( a_i\otimes x_i), \dots, \gamma(a_n\otimes x_n)]\nonumber,
\end{eqnarray}
with $\gamma=Id\otimes\alpha^k\beta^r$. Therefore, $f\otimes\varphi\in C(A\otimes L)$.
\end{proof}

\begin{proposition}
  Let $(L, [\cdot, \dots, \cdot], \varepsilon, \alpha, \beta)$ be a multiplicative $n$-BiHom-Lie color algebra and $t$ an indeterminate.
 Let us define 
$\overline{L}=\{\sum(x\otimes t+y\otimes t^n)| x, y\in L\}, \overline \alpha(L)=\{\sum(\alpha(x)\otimes t+\alpha(y)\otimes t^n)| x, y\in L\}$ and
$\overline\beta(L)=\{\sum(\beta(x)\otimes t+\beta(y)\otimes t^n)| x, y\in L\}$. Then, $(\overline L, [., \dots, .], \varepsilon, \alpha, \beta)$ is a multiplicative 
$n$-BiHom-Lie color algeba with the bracket
$$[x_1\otimes t^{i_1}, x_2\otimes t^{i_2}, \dots, x_n\otimes t^{i_n}]=[x_1, x_2, \dots, x_n]\otimes t^{\sum i_j},$$
for $i_1, \dots, i_n\in\{1, 2, \dots, n\}$. If $k>n$, we let $t^k=0$.
\end{proposition}
\begin{proof}
It follows from a straightforward forward computation.
\end{proof}

For the sake of convenience, we write $xt (xt^n)$ instead of $x\otimes t(x\otimes t^n)$. If $U$ is a subspace of $L$ such that 
$L=U\oplus[L, \dots, L]$, then
$$\overline L=Lt+Lt^n=Lt++Ut^n+[L, \dots, L]t^n.$$
Now define a map $\varphi : QDer(L)\rightarrow End(\overline L)$ by
$$\varphi(D)(at+ut^n+bt^n)=D(a)t+D'(b)t^n, $$
where $D\in QDer(L)$, $D'$ is a map related to $D$ by the definition of quasiderivation, $a\in L, u\in U, b\in[L, \dots, L]$. 

\begin{proposition}
 Let $L, \overline L, \varphi$ be as above. Then
\begin{enumerate}
 \item [(1)] $\varphi$ is even ;
\item [(2)] $\varphi$ is injective and $\varphi(D)$ does not depend on the choise of $D'$ ;
\item [(3)] $\varphi(QDer(L))\subseteq Der(\overline L)$.
\end{enumerate}
\end{proposition}
\begin{proof}
 \begin{enumerate}
  \item [(1)] It follows from the definition of $\varphi$.
\item [(2)] Let $D_1, D_2\in GDer(L)$ such that $\varphi(D_1)=\varphi(D_2)$. Then for all $a\in L, b\in[L, \dots, L]$ and $u\in U$ we have
$$\varphi(D_1)(at+ut^n+bt^n)=\varphi(D_2)(at+ut^n+bt^n),$$
which gives, by the definition of $\varphi$,
$$D_1(a)t+D'_1(b)t^n=D_2(a)t+D_2(b)t^n,$$
so, $D_1(a)=D_2(a)$. Hence $D_1=D_2$, and $\varphi$ is injective.\\
Now, suppose that there exists $D''$ such that
$$\varphi(D)(at+ut^n+bt^n)=D(a)t+D''(b)t^n,$$
and 
$$\sum\varepsilon(d, X_i)[\alpha^k\beta^r(x_1), \dots, D(x_i), \dots, \alpha^k\beta^r(x_n)]=D''([x_1, \dots, x_n]),$$
then we have
$$D'([x_1, \dots, x_n])=D''([x_1, \dots, x_n]),$$
thus $D'(b)=D''(b)$. Hence
$$\varphi(D)(at+ut^n+bt^n)=D(a)t+D'(b)t^n=D(a)t+D''(b)t^n,$$
which implies that $\varphi(D)$ is determined only by $D$.
\item [(3)] We have $[x_1t^{i_1}, x_2 t^{i_2}, \dots, x_n t^{i_n}]=[x_1, x_2, \dots, x_n] t^{\sum i_j}=0$ for all $t^{\sum i_j}\geq n+1$.
Thus, to show that $\varphi(D)\in Der(\overline L)$, we only need to check that the following equality holds :
$$\varphi(D)([x_1t^{i_1}, x_2 t^{i_2}, \dots, x_n t^{i_n}])=
\sum \varepsilon(d, X_i)[\overline\alpha^k\overline\beta^r(x_1t^{i_1}), \dots, \varphi(D)(x_it^{i_1}),
 \dots, \overline\alpha^k\overline\beta^r(x_nt^{i_n})].$$
For arbitrary $x_1, \dots, x_n\in L$, we have
\begin{eqnarray}
 \varphi(D)([x_1t, x_2 t, \dots, x_n t])&=&\varphi(D)([x_1, x_2, \dots, x_n]t^n)=D'([x_1, x_2, \dots, x_n])t^n\nonumber\\
&=&\sum \varepsilon(d, X_i)[\alpha^k\beta^r(x_1), \alpha^k\beta^r(x_2), \dots, \alpha^k\beta^r(x_n)]t^n\nonumber\\
&=&\sum \varepsilon(d, X_i)[\alpha^k\beta^r(x_1t), \alpha^k\beta^r(x_2t), \dots, \alpha^k\beta^r(x_nt)]\nonumber\\
&=&\sum \varepsilon(d, X_i)[\overline\alpha^k\overline\beta^r(x_1t), \dots, \varphi(D)(x_it),
 \dots, \overline\alpha^k\overline\beta^r(x_nt)]\nonumber.
\end{eqnarray}
Therefore, for all $D\in QDer(L)$ we have $\varphi(D)\in Der(\overline L)$.
 \end{enumerate}
\end{proof}

% {\bf Acknowlegments :}

\label{lastpage-01}
\end{document}